\newif\ifFourOR

\FourORfalse 

\ifFourOR 
\RequirePackage{fix-cm}
\documentclass[smallextended]{svjour3}       
\else 
\documentclass[11pt,a4paper]{article}

\usepackage{amsthm}
\fi 

\usepackage{amsmath,amssymb,enumerate,framed,mdwlist,a4wide}
\usepackage{color,colortab}
\usepackage{pgf,tikz}
\usepackage{mathrsfs}
\usetikzlibrary{arrows}
\usepackage[]{todonotes}

\usepackage{multicol}
\usepackage{color}
\usepackage{xspace}

\usepackage{amssymb}
\usepackage{latexsym}
\usepackage{amsfonts}
\usepackage{amsmath}
\usepackage{mathtools}

\usepackage{longtable}
\usepackage{rotating}
\usepackage{lscape}
\usepackage{algpseudocode}
\usepackage{algorithm}
\usepackage{url}
\usepackage[mathscr]{euscript}
\usepackage{hyperref}
\usepackage{pifont} 

\ifFourOR 
\newtheorem{assumption}{Assumption}
\newtheorem{condition}{Condition}
\else 
\usepackage{caption}
\usepackage{footmisc}
\newtheorem{proposition}{Proposition}

\newtheorem{theorem}{Theorem}

\newtheorem{lemma}{Lemma}

\fi 

\newcommand{\Sy}{{\mathscr S}_n}
\newcommand{\sdp}{{\mathscr S}_n^{+}}
\newcommand{\sdn}{{\mathscr S}_n^{-}}

\newcommand{\Aat}{{\mathscr A}^{\top}}

\newcommand{\Aa}{{\mathscr A}}

\newcommand{\RR}{\mathbb{R}}

\newcommand{\ppp}[1]{\tilde{#1}}

\newcommand{\Diag}{{\rm Diag}}

\newcommand{\trace}{{\rm trace}}

\newcommand{\ignore}[1]{}
\newcommand\R{{\mathbb R}}

\def\ignore#1{}

\DeclareMathOperator*{\argmax}{arg\,max}

\newcommand{\cmark}{\ding{51}}
\newcommand{\xmark}{\ding{55}}

\algnewcommand{\Input}[1]{%
    \textbf{Input: }#1}

\allowdisplaybreaks

\ifFourOR 
\usepackage{mathptmx}      
\usepackage{latexsym}
\journalname{4OR}
\else 
\fi 

\begin{document}

\title{Improving ADMMs for Solving Doubly Nonnegative Programs through Dual 
Factorization 
\thanks{This project has received funding from the European Union's Horizon 2020 
research and innovation programme under the Marie Sk\l odowska-Curie grant 
agreement MINOA No 764759 and the Austrian Science Fund (FWF): I 3199-N31.
}
}

\ifFourOR 

\author{Martina Cerulli \and Marianna De Santis \and Elisabeth Gaar \and Angelika Wiegele}

\titlerunning{Improving ADMMs for Solving DNNs through Dual 
Factorization}        


\institute{
Martina Cerulli 
\at CNRS, LIX - Ecole Polytechnique, Institut Polytechnique de Paris, 1 rue Honoré d'Estienne d'Orves, 91120 Palaiseau, France\\
\email{mcerulli@lix.polytechnique.fr}
\and 
Marianna De Santis 
\at Dipartimento di Ingegneria Informatica Automatica e Gestionale, Sapienza Universit\`{a} di Roma, Via Ariosto, 25, 00185 Roma, Italy\\
\email{marianna.desantis@uniroma1.it}
\and 
Elisabeth Gaar 
\at Institut f\"ur Mathematik, Alpen-Adria-Universit\"at Klagenfurt, Universit\"atsstra{\ss}e 65-67, 9020 Klagenfurt, Austria\\
\email{elisabeth.gaar@aau.at}
\and 
Angelika Wiegele 
\at Institut f\"ur Mathematik, Alpen-Adria-Universit\"at Klagenfurt, Universit\"atsstra{\ss}e 65-67, 9020 Klagenfurt, Austria\\
\email{angelika.wiegele@aau.at}
}

\date{Received: date / Accepted: date}

\else 

\author{
Martina Cerulli\footnote{CNRS, LIX - Ecole Polytechnique, Institut Polytechnique de Paris, 1 rue Honoré d'Estienne d'Orves, 91120 Palaiseau, France, {\tt mcerulli@lix.polytechnique.fr}}
\and Marianna De Santis\footnote{Dipartimento di Ingegneria Informatica Automatica e Gestionale, Sapienza Universit\`{a} di Roma, Via Ariosto, 25, 00185 Roma, Italy, {\tt marianna.desantis@uniroma1.it}}
\and Elisabeth Gaar\footnote{Institut f\"ur Mathematik, Alpen-Adria-Universit\"at Klagenfurt, Universit\"atsstra{\ss}e 65-67, 9020 Klagenfurt, Austria, {\tt elisabeth.gaar@aau.at}}
\and Angelika Wiegele\footnote{Institut f\"ur Mathematik, Alpen-Adria-Universit\"at Klagenfurt, Universit\"atsstra{\ss}e 65-67, 9020 Klagenfurt, Austria, {\tt angelika.wiegele@aau.at}}
}

\date{\today}

\fi 

\maketitle

\begin{abstract} 
\ifFourOR 
\par\noindent
\else 
\fi 

Alternating direction methods of multipliers (ADMMs) are popular approaches to
handle large scale semidefinite programs that gained attention during the past decade.
In this paper, we focus on solving doubly nonnegative programs (DNN), which are semidefinite programs where the elements of the matrix variable 
are constrained to be nonnegative. Starting from two algorithms already proposed in
the literature on conic programming, we introduce two new ADMMs
by employing a factorization of the dual variable. 

It is well known that first order methods are not suitable to compute high precision optimal solutions, however an optimal solution of moderate precision often suffices to get high quality lower bounds on the primal optimal objective function value. 
We present methods to obtain such bounds by either perturbing the dual objective function value or by constructing a dual feasible solution from a dual approximate optimal solution.
Both procedures can be used as a post-processing phase in our ADMMs.

Numerical results for DNNs that are relaxations of the stable set problem are presented. They show the impact of using the factorization of the dual variable in order to improve the progress towards the optimal solution within an iteration of the ADMM. This decreases the number of iterations as well as the CPU time to solve the DNN to a given precision. 
The experiments also demonstrate that within a computationally cheap post-processing, we can compute bounds that are close to the optimal value even if the DNN was solved to moderate precision only. 
This makes ADMMs applicable also within a branch-and-bound algorithm.
\end{abstract}

\section{Introduction}\label{sec:intro}
In a semidefinite program (SDP) one wants to find a positive semidefinite (and 
hence symmetric) matrix such that linear -- in the entries of the matrix -- 
constraints are fulfilled and a linear objective function is minimized. If the 
matrix is also required to be entrywise nonnegative, the problem is called 
doubly nonnegative program (DNN). Since interior point methods fail (in terms 
of time and memory required) when the scale of the SDP is big, augmented 
Lagrangian approaches became more and more popular to solve this class of 
programs. 
Wen, 
Goldfarb 
and Yin~\cite{Wen2010} as well as Malick, Povh, Rendl and Wiegele~\cite{MaPoReWi:09} and De~Santis, Rendl and 
Wiegele~\cite{deSaReWie:2018} 
considered alternating direction methods of 
multipliers (ADMMs) to solve SDPs.
One can directly apply these ADMMs to solve DNNs, too, by
introducing nonnegative slack 
variables for the nonnegativity constraints in order to obtain equality constraints only.
However, this increases the size of the problem significantly.

In this paper, we first present two ADMMs already proposed in the literature 
(namely \texttt{ConicADMM3c} by Sun, Toh and Yang~\cite{Sun2015AC3} and \texttt{ADAL+}~\cite{Wen2010}) to specifically solve  DNNs. Then
we introduce two new methods: \texttt{DADMM3c}, which employs a factorization of the dual matrix to avoid spectral 
decompositions, and \texttt{DADAL+} taking advantage of the practical benefits of \texttt{DADAL}~\cite{deSaReWie:2018}. Note that there are examples 
for which a 3-block ADMM (like \texttt{DADAL+}) diverges. However, the question of convergence of 3-block ADMMs for SDP relaxations arising from combinatorial optimization problems is still open.

In case the DNN is used as relaxation of some combinatorial optimization problem, one is interested in dual bounds, i.e.\ bounds that are the dual objective function value of a dual feasible solution. 
In case of a minimization problem this is a lower bound, in case of a maximization problem an upper bound.
Having bounds is in particular important if one intends to use the relaxation within a branch-and-bound algorithm. 
This, however, means that one needs to solve the DNN to high precision such that the dual solution is feasible and hence the dual objective function value is a reliable bound. Typically, first order methods can compute solutions of moderate precision in reasonable time, whereas progressing to higher precision can become expensive. 
To overcome this drawback, we present two methods to compute a dual bound from a solution obtained by the ADMMs within a post-processing phase. 

In the following section we state our notations and introduce the 
formulation of standard primal-dual SDPs and DNNs. In Section~\ref{sec:admm} we 
go through the two existing ADMMs for DNNs we mentioned before, and in 
Section~\ref{sec:Zfactorization} we introduce the tool of dual matrix 
factorization used in the new ADMMs \texttt{DADAL+} and \texttt{DADMM3c} presented later in 
the same section. 
In Section~\ref{sec:SafeBounds} we present two methods for obtaining dual bounds from a solution of a DNN that satisfies the optimality criteria to moderate precision only. 
Section~\ref{sec:numres} shows numerical results for 
instances of DNN relaxations of the stable set problem. 
We evaluate the impact of the dual factorization within the methods as well as the two post-processing schemes for obtaining dual bounds.  
Section~\ref{sec:conclusions} concludes the paper.

\subsection{Problem Formulation and Notations}
Let $\Sy$ be the set of $n$-by-$n$ symmetric matrices,  $\sdp~\subset~\Sy$ be the set of positive semidefinite matrices and $\Sy^{-}~\subset~\Sy$ be the set of negative semidefinite matrices.
Denoting by $\left\langle X,Y\right\rangle = \trace(XY)$ the standard inner product in $\Sy$,
we write the standard primal-dual pair of SDPs as 
\begin{equation}
\begin{aligned}
&\min \,       && \left\langle C,X\right\rangle\\
&\mbox{ s.t. } &&  \Aa X = b \\
&              && X\in \sdp
\end{aligned}
\label{eq:primal}
\end{equation}
and
\begin{equation}
\begin{aligned}
&\max \,       && b^Ty\\
&\mbox{ s.t. } && \Aat y +Z = C \\
&              && Z\in \sdp,
\end{aligned}
\label{eq:dual}
\end{equation}
where $C\in\Sy$, $b\in\RR^m$, $\Aa: \Sy \rightarrow \RR^m$ is the linear operator 
$(\Aa X)_i = \left\langle A_i,X\right\rangle$ with $A_i\in\Sy$, $i=1,\ldots,m$ and 
$\Aat: \RR^m \rightarrow \Sy$ is its adjoint operator, so $\Aat y = \sum_i y_i A_i$
for $y \in \RR^m$.

When in the primal SDP~\eqref{eq:primal} the elements of $X$ are constrained to be nonnegative, then the SDP is called a doubly nonnegative program (DNN). 
To be more precise the primal DNN is given as 
\begin{equation}
\begin{aligned}
&\min  \,      && \left\langle C,X\right\rangle\\
&\mbox{ s.t. } && \Aa X = b \\
&              && X\in \sdp, \quad X\geq 0.
\end{aligned}
\label{eq:primaldnn}
\end{equation}
Introducing $S$ as the dual variable related to the 
nonnegativity constraint $X\geq 0$, we write the dual of the DNN~\eqref{eq:primaldnn} as 
\begin{equation}
\begin{aligned}
&\max  \,      && b^Ty\\
&\mbox{ s.t. } && \Aat y + Z + S = C \\
&              && Z\in \sdp, \quad S \in \Sy, \quad  S\geq 0.
\end{aligned}
\label{eq:dualdnn}
\end{equation}

We assume that both the primal DNN~\eqref{eq:primaldnn} and the dual DNN~\eqref{eq:dualdnn} have strictly feasible points (i.e. Slater's condition is satisfied), so strong duality holds.
Under this assumption, $(y, S, Z, X)$ is optimal for~\eqref{eq:primaldnn} and~\eqref{eq:dualdnn} if and only if
\begin{equation}\label{eq:optCond}
\begin{aligned}
  \Aa X &= b, \quad& \Aat y +Z + S &= C,\quad&   ZX &= 0, \\
  X &\in \sdp, \quad& Z &\in \sdp, \quad& \left\langle S,X\right\rangle &= 0, \\
  X &\geq 0, \quad& S \in \Sy, \quad S&\geq 0,\\ 
\end{aligned}
\end{equation}
hold. We further assume that the constraints formed through the operator $\Aa$ are linearly independent.

Let $v\in \RR^n$ and $M\in \RR^{m\times n}$. In the following,
$M(i,:)$ is defined as the i-th row of $M$ and $M(:,j)$ as the j-th column of $M$.
Further we denote by $\Diag{(v)}$ the diagonal matrix having 
$v$ on the main diagonal. The vector $e_i$ is defined as the $i$-th vector of the standard basis in $\RR^n$. 
Whenever a norm is used, we consider the Frobenius norm in case of matrices and the Euclidean norm 
in case of vectors.
Let $S \in \Sy$. 
We denote the projection of $S$ onto
the positive semidefinite and negative semidefinite cone by $(S)_{+}$ and $(S)_{-}$, respectively.
The projection of $S$ onto the nonnegative orthant is denoted by $(S)_{\ge 0}$. 
Moreover we denote by $\lambda(S)$ the vector of the eigenvalues of $S$ and by $\lambda_{\min}(S)$ and $\lambda_{\max}(S)$ the smallest and largest eigenvalue of $S$, respectively.

\section{ADMMs for Doubly Nonnegative Programs}\label{sec:admm}
In this section, we present two different ADMMs for solving DNNs.
Let $X\in \Sy$ be the Lagrange multiplier for the dual equation 
$\Aat y + Z + S -C = 0$ and $\sigma>0$ be fixed. 
Then the augmented Lagrangian of the dual DNN~\eqref{eq:dualdnn} is defined as
$$
L_\sigma(y,S,Z; X) = b^Ty - \langle \Aat y + Z + S - C, X\rangle -
\frac{\sigma}{2}\|\Aat y + Z + S - C\|^2.$$
In the classical augmented Lagrangian method applied to the dual DNN~\eqref{eq:dualdnn}
the problem  
\begin{equation}\label{eq:maxLagZ}
 \begin{aligned}
      &\max  \, && L_\sigma(y,S,Z; X) \\
      &\mbox{ s.t. } && y\in \RR^m, \quad S \in \Sy, \quad S\geq0, \quad Z\in \sdp,
\end{aligned}
\end{equation}
where $X$ is fixed and $\sigma>0$ is a penalty parameter is addressed at every iteration.

Once Problem~\eqref{eq:maxLagZ} is (approximately) solved, the multiplier $X$ is updated by the first order rule
\begin{equation}\label{eq:firstOrderUpdate}
 X = X + \sigma (\Aat y + Z +S - C)
\end{equation}
and the process is iterated until convergence, i.e., until the optimality conditions~\eqref{eq:optCond} are satisfied
within a certain tolerance (see~\cite[Chapter 2]{Be:82} for further details).

If the augmented Lagrangian $L_\sigma(y,S,Z; X)$ is maximized 
with respect to $y$, $S$ and $Z$ not simultaneously but one after the other, this yields the well known
alternating direction method of multipliers (ADMM). 
The number of blocks of an ADMM is the number of blocks of variables for which Problem~\eqref{eq:maxLagZ} is maximized separately, so we consider a 3-block ADMM.
Such an ADMM has been specialized and used by 
Wen, Goldfarb and Yin~\cite{Wen2010} to address DNNs and in the following we refer to this method as
\texttt{ADAL+}. Details will be given in Section~\ref{sec:adal+}.
Even though in all our numerical tests this algorithm reaches the desired precision of our stopping criteria, it has been recently shown in~\cite{chen2016direct}
that an ADMM with more than two blocks may diverge. 

In order to overcome this theoretical issue, Sun, Toh and Yang~\cite{Sun2015AC3}
proposed to update the third block twice per iteration, or, in other words, to maximize $L_\sigma(y,S,Z; X)$
with respect to the variable $y$ two times in one iteration.
Their algorithm, named \texttt{ConicADMM3c} and detailed in Section~\ref{sec:ConicADMM3c}, is the first theoretically convergent 3-block ADMM proposed in the context of conic programming.

\subsection{\texttt{ADAL+}}\label{sec:adal+}
In the following, we refer to the ADMM presented by Wen, Goldfarb and Yin in~\cite{Wen2010} and applied to the dual DNN~\eqref{eq:dualdnn} as \texttt{ADAL+}.
As already mentioned, \texttt{ADAL+} iterates the maximization of the augmented Lagrangian with respect to each block of dual variables. 
To be more precise the new point $(y^{k+1},S^{k+1},Z^{k+1},X^{k+1})$ is computed by the following steps:
\begin{align} \label{eq:opty}
 y^{k+1} &= \argmax_{y\in \RR^m} L_{\sigma^k}(y,S^k,Z^{k}; X^{k}), \\
 \label{eq:optS}
 S^{k+1} &= \argmax_{S \in \Sy, S\geq 0} L_{\sigma^k}(y^{k+1},S,Z^k; X^{k}), \\
\label{eq:optZ}
 Z^{k+1} &= \argmax_{Z\in \sdp} L_{\sigma^k}(y^{k+1},S^{k+1},Z; X^{k}),\\
 \label{eq:updateX}
 X^{k+1} &= X^{k} + \sigma^k ( \Aat y^{k+1} + Z^{k+1} +S^{k+1} - C).
\end{align}
\\
The update of $y$ in~\eqref{eq:opty} is derived from the first-order optimality condition of the problem on the right-hand side of~\eqref{eq:opty},  so 
$y^{k+1}$ 
is the unique solution of
\[
\nabla_y L_{\sigma^k} (y,S^{k},Z^{k}; X^{k}) = b - \Aa( X^k + \sigma^k (\Aat y + Z^k +S^k - C)) = 0,
\]
that is
\[
y^{k+1}= (\Aa\Aat)^{-1}\Big(\frac{1}{\sigma^k} b - \Aa(\frac{1}{\sigma^k} X^k + Z^k +S^k - C)\Big).
\]
As shown in~\cite{Wen2010}, the update of $S$ according to~\eqref{eq:optS} 
is equivalent to 
\[ \min_{S\in \Sy, S\geq 0} \|S - U^{k+1}\|^2,\]
where $U^{k+1} = C - \Aat y^{k+1} -Z^k -\frac{1}{\sigma^k} X^k$. Hence, $S^{k+1}$ is obtained as the projection 
of $U^{k+1}$ onto the nonnegative orthant, namely
\[S^{k+1}= \big(U^{k+1}\big)_{\ge 0} = 
\big(C - \Aat y^{k+1} -Z^k -\frac{1}{\sigma^k} X^k\big)_{\ge 0}.\]
Then, the update of $Z$ in~\eqref{eq:optZ} is conducted by considering the equivalent problem
\begin{equation}\label{eq:proj}
 \min_{Z\in \sdp} \|Z + W^{k+1}\|^2,
\end{equation}
with $W^{k+1} = (\frac{1}{\sigma^k}X^k - C + \Aat y^{k+1} +S^{k+1})$, or, in other words, by projecting $W^{k+1}\in \Sy$
onto the (closed convex) cone $\sdn$ and taking its additive inverse (see Algorithm~\ref{alg:ADAL+}). 
Such a projection is computed via the spectral decomposition of the matrix $W^{k+1}$.

Finally, it is easy to see that the update of $X$ in~\eqref{eq:updateX} can be performed considering the projection of $W^{k+1}\in \Sy$ 
onto $\sdp$ multiplied by $\sigma^k$, namely
\begin{align*}
 X^{k+1} & = X^k + \sigma^k (\Aat y^{k+1} + Z^{k+1} +S^{k+1} - C ) = \\
 & = \sigma^k(X^k/\sigma^k - C + \Aat y^{k+1} +S^{k+1} - (X^k/\sigma^k - C + \Aat y^{k+1} +S^{k+1})_-) = \\
 & = \sigma^k(X^k/\sigma^k - C + \Aat y^{k+1} +S^{k+1})_+.
\end{align*}
We report in Algorithm~\ref{alg:ADAL+} the scheme of \texttt{ADAL+}.
\begin{algorithm}[ht]
    \caption{Scheme of \texttt{ADAL+} from~\cite{Wen2010}}
    \label{alg:ADAL+}
    \begin{algorithmic}[1]
        \State Choose $\sigma>0$, $\varepsilon >0$, $X\in \sdp$, $Z\in \sdp$, 
        $S \in \Sy$ with $S\geq 0$
        \State $\delta = \max \{ r_P, r_D, r_{PP}, r_{CS}\}$
        \While{$\delta > \varepsilon$}
        \State $y= (\Aa\Aat)^{-1}\Big(\frac{1}{\sigma} b - \Aa(\frac 1 
        \sigma X -C + Z +S)\Big)$ \label{yupdate}
        \State $S= (C - \Aat y -Z -\frac 1 \sigma X)_{\ge 0}$ 
        \State $Z =-(X/\sigma -C + \Aat y +S)_-$ and $ X 
        =\sigma(X/\sigma -C + \Aat y +S)_+$
        \State $\delta = \max \{ r_P, r_D, r_{PP}, r_{CS}\}$
        \State Update $\sigma$ 
        \EndWhile
    \end{algorithmic}
\end{algorithm}

%
The stopping criterion of \texttt{ADAL+} considers the following errors
%
\begin{equation*}
    \begin{aligned}
 r_P&= \frac{\|\Aa X - b\|}{1+\|b\|}, \quad \quad &
 r_D&= \frac{\|\Aat y + Z + S - C\|}{1+ \|C\|}, \\
r_{PP} &= \frac{\|X - (X)_{\ge 0}\|}{1+\|X\|}, \quad\quad &
 r_{CS} &= \frac{|\left\langle S,X\right\rangle |}{1+\|X\|+\|S\|},
    \end{aligned}
\end{equation*}
related to primal feasibility ($\Aa X = b$, $X\geq 0$), dual feasibility ($\Aat y + Z + S = C$)
and complementarity  
condition ($\left\langle S,X\right\rangle = 0$).
More precisely,
the algorithm stops as soon as the quantity
\[
\delta = \max \{ r_P, r_D, r_{PP}, r_{CS}\}
\]
is less than a fixed precision $\varepsilon>0$.
 
The other optimality conditions (namely $X\in \sdp$, $Z\in \sdp$, $S \in \Sy$, $S\geq0$, $ZX=0$) 
are satisfied up to machine accuracy throughout the algorithm thanks to the projections employed in \texttt{ADAL+}.

\subsection{\texttt{ConicADMM3c}}\label{sec:ConicADMM3c}
A major drawback of \texttt{ADAL+} is that it is not necessarily convergent.
By considering two updates of the variable $y$ within one iteration, Sun, Toh and Yang are able to prove that the
algorithm \texttt{ConicADMM3c} proposed in~\cite{Sun2015AC3} and detailed in Algorithm~\ref{alg:ConicADMM3c} 
is a 3-block convergent ADMM: 
Under certain assumptions, they show that the sequence $\{(y^k, S^k, Z^k; X^k)\}$ 
produced by \texttt{ConicADMM3c} converges to a KKT point of the primal DNN~\eqref{eq:primaldnn} and the dual DNN~\eqref{eq:dualdnn}. 
Note that also the order of the updates on the blocks of variables is different with respect to \texttt{ADAL+}. 
The convergence analysis is based on the fact  
that \texttt{ConicADMM3c} is equivalent to a  semi-proximal ADMM.

With respect to \texttt{ADAL+}, \texttt{ConicADMM3c} has the drawback that fewer 
optimality conditions are satisfied up to machine accuracy throughout the algorithm.
Additionally to $r_P, r_D, r_{PP}$ and $r_{CS}$, the stopping criterion of \texttt{ConicADMM3c} has to take into account the errors
\[ r_{PD} =  \frac{\|(-X)_+\|}{1+\|X\|} \quad \text{ and } \quad  r_{CZ} =  \frac{\|\left\langle Z,X\right\rangle \|}{1+\|X\|+\|Z\|},\]
related to the primal feasibility $X\in \sdp$ and the complementarity condition $ZX = 0$. 
In fact, as the second update of $y$ is performed after the update of $Z$, the spectral decomposition 
of $W^{k+1}$ cannot be used to update $X$ as in \texttt{ADAL+} and both the complementarity condition $ZX=0$ and the positive semidefiniteness of $X$ 
are not satisfied by construction.
(We will give a summary on the conditions satisfied throughout the algorithms in Table~\ref{tab:opt_cond_true} in a subsequent section.)
From a computational point of view this slows down the convergence of the scheme, which will be confirmed in our computational evaluation in Section~\ref{sec:numres}. 

\begin{algorithm}[ht]
    \caption{Scheme of \texttt{ConicADMM3c} from~\cite{Sun2015AC3}}
    \label{alg:ConicADMM3c}
    \begin{algorithmic}[1]
        \State Choose $\sigma>0$, $\varepsilon >0$, $X\in \sdp$, $Z\in \sdp$, 
        $S \in \Sy$ with $S\geq 0$ 
        \State $\delta = \max \{ r_P, r_D, r_{PP}, r_{PD}, r_{CS}, 
        r_{CZ}\}$
        \While{$\delta > \varepsilon$}
        \State $Z =-(X/\sigma -C + \Aat y +S)_-$
        \State $ y= (\Aa\Aat)^{-1}\Big(\frac{1}{\sigma} b - \Aa(\frac 1 
        \sigma X -C + Z +S)\Big)$ 
        \State $S= (C - \Aat y -Z -\frac 1 \sigma X)_{\ge 0}$ 
        \State $ y= (\Aa\Aat)^{-1}\Big(\frac{1}{\sigma} b - \Aa(\frac 1 
        \sigma X -C + Z +S)\Big)$ 
        \State  $X = X + \sigma (\Aat y + Z + S -C)$
        \State $\delta = \max \{ r_P, r_D, r_{PP}, r_{PD}, r_{CS}, 
        r_{CZ}\}$
        \State Update $\sigma$ 
        \EndWhile
    \end{algorithmic}
\end{algorithm}

\section{Dual Matrix Factorization}\label{sec:Zfactorization}
In this section, we present our new variants of \texttt{ADAL+} and \texttt{ConicADMM3c}, namely \texttt{DADAL+} and \texttt{DADMM3c}, 
where a factorization of the dual variable $Z$ is employed.
We adapt the method introduced by De~Santis, Rendl and Wiegele in~\cite{deSaReWie:2018}.
In particular, we look at the augmented Lagrangian problem where the positive semidefinite constraint on the dual matrix $Z$ is eliminated by considering the factorization $Z=VV^\top$. 
To be more precise, in each iteration of the ADMMs for fixed $X$, we focus on 
the problem
\begin{equation}\label{eq:maxLagV}
 \begin{aligned}
&\max \, && L_\sigma (y,S,V; X)\\
&\mbox{ s.t. } && y\in \RR^m, \quad S \in \Sy, \quad S\geq 0, \quad V\in  \RR^{n \times r},
\end{aligned}
\end{equation}
where \[L_\sigma (y,S,V; X) =  b^Ty - \langle \Aat y +  VV^\top + S - C, X\rangle -
\frac{\sigma}{2}\|\Aat y + VV^\top +S - C \|^2.\]
Compared to~\eqref{eq:maxLagZ} the constraint $Z \in \sdp$ is replaced by $Z=VV^\top$ 
for some $V\in  \RR^{n \times r}$, so $Z \in  \sdp$ is fulfilled automatically.
Note that the number of columns $r$ of the matrix $V$ represents the rank of $Z$.

The use of the factorization of the dual variable in \texttt{ADAL+} should improve the numerical performance of the algorithm 
when dealing with structured DNNs, as it happens in the comparison of the algorithm \texttt{DADAL} with \texttt{ADAL} when dealing 
with structured SDPs~\cite{deSaReWie:2018}.
For what concerns \texttt{ConicADMM3c}, we will see in Section~\ref{sec:DADMM3c} that using the factorization of the dual variable 
allows to avoid any spectral decomposition 
along the iterations of the algorithm, without compromising the theoretical convergence of the method.

Note that Problem~\eqref{eq:maxLagV} is unconstrained with respect to the variables $y$ and $V$. In particular, the following holds.
\begin{proposition}\label{prop:NecOptCond}
Let $(y^*, S^*, V^*)\in \RR^m \times \Sy \times \RR^{n \times r}$ be a stationary point 
of~\eqref{eq:maxLagV}, then
\begin{equation}\label{eq:OptCondLag}
 \begin{aligned}
  \nabla_y L_\sigma (y^*,S^*,V^*; X) &= b - \Aa( X + \sigma (\Aat y^* + {V^*V^*}^\top + S -C)) = 0
  \text{ and }\\
  \nabla_V L_\sigma (y^*,S^*,V^*; X) &= -2 (X + \sigma (\Aat y^* + {V^*V^*}^\top + S-C))V^* = 0.
 \end{aligned}
\end{equation}
\end{proposition}
Proposition~\ref{prop:NecOptCond} implies that fulfilling the necessary optimality conditions with respect to $y$
is equivalent to solve one system of linear equations.

As in~\cite{deSaReWie:2018}, we consider Algorithm~\ref{alg:OneItUncMax} 
in order to update $y$ and $V$ (and hence $Z$) for fixed $S$ and $X$. 
In particular in Algorithm~\ref{alg:OneItUncMax}, starting from $(y,S,V; X)$, we move 
$V$ along an ascent direction $D_V\in  \RR^{n \times r}$ with a stepsize $\alpha$.
While doing this, we update $y$ in such a way that we keep its optimality 
conditions of~\eqref{eq:maxLagV} satisfied, so  
$\nabla_y L_\sigma (y,S,V+\alpha D_V; X) = 0$ holds for the updated $y$
(see~\cite[Proposition~2]{deSaReWie:2018}). 
We stop as soon as the necessary optimality 
conditions with respect to $V$ (see Proposition~\ref{prop:NecOptCond}) are fulfilled
to a certain precision.

As in the algorithm \texttt{DADAL} presented in~\cite{deSaReWie:2018}, in our implementation we set $D_V$
either to the gradient of $L_\sigma (y,S,V; X)$ or to the gradient scaled with 
the inverse of the diagonal of the Hessian of $L_\sigma (y,S,V; X)$.
In order to determine a stepsize $\alpha$, at Step~\ref{alg:linesearch} in Algorithm~\ref{alg:OneItUncMax} 
we could perform an exact linesearch to maximize 
$L_\sigma (y(V + \alpha D_{V}),S, V + \alpha D_{V}; X)$ with respect to $\alpha$. 
This is a polynomial of degree 4 
in $\alpha$, so we can interpolate it from five different points in order to get its analytical expression and by this determining the maximizer explicitly. 
In practice we evaluate  $L_\sigma(y(V + \alpha D_{V}),S, V + \alpha D_{V}; X)$ for 1000 different values of $\alpha \in (0,10)$ and take the $\alpha$ 
corresponding to the maximum value of $L_\sigma$. 
 
\begin{algorithm}[ht]
    \caption{Update of $(y,V)$ for factorization $Z=VV^\top$}
    \label{alg:OneItUncMax}
    \Input{$\sigma>0$, $X \in \sdp$, $y\in \RR^m$, $V\in \RR^{n\times r}$, 
        $S \in \Sy$ with $S\geq 0$}
    \begin{algorithmic}[1]
        \State Choose $\varepsilon_{inner}>0$
        \While{$\| \nabla_{V} L_{\sigma}(y,S,V; X) \|<\varepsilon_{inner}$}
        \State Compute ascent direction $D_V\in
        \RR^{n\times r}$ 
        \State Compute stepsize $\alpha$ \label{alg:linesearch}
        \State  $y = y(V + \alpha D_V)$ and $V = V + \alpha D_V$
        \EndWhile
    \end{algorithmic}
\end{algorithm}

As output of Algorithm~\ref{alg:OneItUncMax}, we get 
$y$ and $V$ (and therefore also $Z = VV^\top$) that have been updated through the maximization
of the augmented Lagrangian~\eqref{eq:maxLagV} with respect to $V$. This leads to a new point $(y,S,V; X)$. 

This update can be used within \texttt{ADAL+} and \texttt{ConicADMM3c} as detailed in the following.

\subsection{\texttt{DADAL+}}
First we consider \texttt{DADAL+}, our version of \texttt{ADAL+} where the use of the factorization of the dual variable~$Z$ leads to a double update of $Z$. 
As a further enhancement of the algorithm \texttt{ADAL+} devised in~\cite{Wen2010}, we 
propose to perform also a double update of the dual variable $y$.

To be more precise, we replace the first update of $y$ in \texttt{ADAL+} with a update 
of $y$ and $V$ with Algorithm~\ref{alg:OneItUncMax} in \texttt{DADAL+}.
Furthermore in \texttt{DADAL+} we update $y$ not only before, but also a second time after the computation of $S$.
This second update is performed by 
applying the closed formula solution of the maximization problem in~\eqref{eq:opty}.
Note that the second update of $y$ is performed before the update of $Z$ so that by computing
the spectral decomposition of $W = X/\sigma -C +\Aat y + S$, we can simultaneously update $Z$ and $X$ and both the 
complementarity condition $ZX=0$ and the positive semidefiniteness of $X$ are satisfied
 up to machine accuracy throughout the algorithm in the same way it is the case in \texttt{ADAL+}.
The scheme of \texttt{DADAL+} is detailed in Algorithm~\ref{alg:BPMComb}.
\begin{algorithm}[ht]
    \caption{Scheme of \texttt{DADAL+}}
    \label{alg:BPMComb}
    \begin{algorithmic}[1]
        \State Choose $\sigma>0$, $r>0$, $\varepsilon >0$, $X\in \sdp$, $S \in 
        \Sy$ with $S\geq 0$, $V\in 
        \R^{n\times r}$, $y\in \R^m$
        \State $Z=VV^{\top}$
        \State $\delta = \max \{ r_P, r_D, r_{PP}, r_{CS}\}$
        \While{$\delta > \varepsilon$}
        \State Update $(y,V)$ with Algorithm~\ref{alg:OneItUncMax}
        \State $Z = VV^\top$
        \State $S =  (C - \Aat y -Z -\frac 1 \sigma X)_{\ge 0} $
        \State $y= (\Aa\Aat)^{-1}\Big(\frac{1}{\sigma} b - \Aa(\frac 1 
        \sigma X -C + Z +S)\Big)$ 
        \State $Z = - (X/\sigma -C +\Aat y  + S)_-$ and $X 
        =\sigma(X/\sigma - C + \Aat y   + S)_+$
        \State $r = \mathrm{rank(Z)}$
        \State Update $V$ such that $VV^\top = Z$
        \State $\delta =  \max \{ r_P, r_D, r_{PP}, r_{CS}\}$
        \State Update $\sigma$ 
        \EndWhile
    \end{algorithmic}
\end{algorithm}

\subsection{\texttt{DADMM3c}}\label{sec:DADMM3c}
We now investigate the use of the dual factorization within the 
algorithm \texttt{ConicADMM3c} and call the modified algorithm \texttt{DADMM3c}.
In \texttt{ConicADMM3c}, the effort spent to compute the spectral decomposition of $W=X/\sigma -C +\Aat y + S$ is not that well exploited 
as it is used to update only the dual matrix $Z$ but not the primal matrix $X$.
Hence in \texttt{DADMM3c} we update $Z$ and $y$ by employing the factorization
$Z = VV^\top$ and performing Algorithm~\ref{alg:OneItUncMax}
instead of updating them by a spectral decomposition and a closed formula as it is done in \texttt{ConicADMM3c}.
Note that Algorithm~\ref{alg:OneItUncMax} is able to compute stationary points of Problem~\eqref{eq:maxLagV}, that 
are not necessarily global optima.
However, assuming that the update of $y$ and $V$ at Step~\ref{stepFact} in Algorithm~\ref{alg:DADMM3c} is done such that Problem~\eqref{eq:maxLagV} is solved to optimality, the theoretical convergence of the method is maintained.
Note that the computation of any spectral decomposition is avoided. 
The scheme of the algorithm \texttt{DADMM3c} is detailed in 
Algorithm~\ref{alg:DADMM3c}.
\begin{algorithm}[ht]
    \caption{Scheme of \texttt{DADMM3c}}
    \label{alg:DADMM3c}
    \begin{algorithmic}[1]
        \State Choose $\sigma>0$, $r>0$, $\varepsilon >0$,
        $X\in \sdp$, $S \in \Sy$ with $S\geq 0$, $V\in \R^{n\times r}$, $y\in 
        \R^m$
        \State $Z=VV^{\top}$
        
        \State $\delta = \max \{ r_P, r_D, r_{PP}, r_{PD}, r_{CS}, 
        r_{CZ}\}$
        \While{$\delta > \varepsilon$}
        \State Update $(y,V)$ with Algorithm~\ref{alg:OneItUncMax} \label{stepFact}
        \State $Z = VV^\top$
        \State $S= (C - \Aat y -Z -\frac 1 \sigma X)_+$ 
        \State $y= (\Aa\Aat)^{-1}\Big(\frac{1}{\sigma} b - \Aa(\frac 1 
        \sigma X -C + Z +S)\Big)$
        \State $X = X + \sigma (Z + S + \Aat y -C)$
        \State $\delta = \max \{ r_P, r_D, r_{PP}, r_{PD}, r_{CS}, 
        r_{CZ}\}$
        \State Update $\sigma$ 
        \EndWhile
    \end{algorithmic}
\end{algorithm}

A limit of \texttt{DADMM3c} is that the rank of $Z$ is not updated throughout the iterations.
This means that the maximization of $L(y,S,V;X)$ with respect to $V$ is performed keeping $r$ fixed to the initial value that
in our implementation is $n$.
It is still an open question to update the rank of $Z$ in a beneficial way.

On the other hand, note that in \texttt{DADAL+} the rank of $Z$ is determined at every iteration 
through the eigenvalue decomposition in the second update of $Z$.

As already mentioned in Section~\ref{sec:admm}, some of the optimality conditions are satisfied throughout the algorithms \texttt{ADAL+}/\texttt{DADAL+} and \texttt{ConicADMM3c}/\texttt{DADMM3c}. A summary is presented in Table~\ref{tab:opt_cond_true}.

\begin{table}[ht]
    \centering
    \begin{tabular}{l|ccc|ccc|cc}
        & \rotatebox{90}{$\Aa X = b$} & \rotatebox{90}{$X \in \sdp$} & \rotatebox{90}{$X\ge 0$} & \rotatebox{90}{$\Aat y + Z + S = C$} & \rotatebox{90}{$Z \in \sdp$} & \rotatebox{90}{$S\in \Sy$, $S\ge 0$} & \rotatebox{90}{$ ZX = 0$} & \rotatebox{90}{$\langle S,X \rangle = 0$}\\
        \hline
        \begin{tabular}{l} \texttt{ADAL+} \\ \texttt{DADAL+}\end{tabular} & \xmark &  \cmark & \xmark & \xmark & \cmark & \cmark & \cmark & \xmark \\
        \hline
        \begin{tabular}{l} \texttt{ConicADMM3c} \\ \texttt{DADMM3c} \end{tabular} & \xmark & \xmark & \xmark & \xmark & \cmark & \cmark & \xmark & \xmark \\
        \hline
    \end{tabular}
    \caption{Optimality conditions~\eqref{eq:optCond} satisfied through the algorithms by construction are indicated by a checkmark, all others by an x-mark.}
    \label{tab:opt_cond_true}
\end{table}

\section{Computation of Dual Bounds}
\label{sec:SafeBounds}
When solving combinatorial optimization problems, DNN relaxations very often yield high quality bounds. 
These bounds can then be used within a branch-and-bound framework in order to get an exact solution method.
In this section we want to discuss how we can obtain lower bounds on the optimal objective function value of the primal DNN~\eqref{eq:primaldnn} from a dual solution of moderate precision only.

Thanks to weak and strong duality results, the objective function value of every feasible solution of the dual DNN~\eqref{eq:dualdnn} is a lower bound on the optimal objective function value of the primal DNN~\eqref{eq:primaldnn} and the optimal values of the primal and the dual DNN  coincide. 
Therefore every dual feasible solution and in particular the optimal dual solution give rise to a dual bound.

Note that the dual objective function value serves as a dual bound only if the DNN relaxation is solved to high precision. If the DNN is solved to moderate precision, the dual objective function value might not be a bound as the dual solution might be infeasible. However, solving the DNN to high precision comes with enormous computational costs.

So unfortunately \texttt{ADAL+}, \texttt{DADAL+}, \texttt{ConicADMM3c} and
\texttt{DADMM3c} are not suitable to produce a bound fast.
Running an ADMM typically gives approximate optimal solutions rather quickly, while going to optimal solutions with high precision can be very time consuming. 
As the dual constraint $\Aat y + Z + S -C = 0$ does not necessarily hold in every iteration of the four algorithms (see Table~\ref{tab:opt_cond_true}), obtaining a dual feasible solution 
with sufficiently high precision with ADMMs may take extremely long.

To save time, but still ensure that we obtain a dual bound, we will stop the four methods at a certain precision. After that we will use one of two procedures in a post-processing phase in order to obtain a bound. In Section~\ref{sec:errorBound} we will describe how to obtain a bound with a method already presented in the literature.
In Section~\ref{sec:nightjet} we present a new procedure for obtaining a dual feasible solution and hence a bound from an approximate optimal solution.

\subsection{Dual Bounds through Error Bounds}
\label{sec:errorBound}
In this section we present the method to obtain lower bounds on the primal 
optimal value of an SDP of the form~\eqref{eq:primal} introduced by Jansson, Chaykin and 
Keil~\cite{JaChayKeil2007}. We adapt this method for DNNs in order to use it 
in a post-processing phase of the four ADMMs presented above. 
We start with the following lemma from~\cite[Lemma 3.1]{JaChayKeil2007}.

\begin{lemma}Let $Z$, $X$ be symmetric matrices of dimension $n$ that satisfy
\begin{equation}\label{eq:upplow}
    \underline{z} \leq \lambda_{\min}(Z), \quad 0 \leq \lambda_{\min}(X), \quad \lambda_{\max}(X) \leq \bar{x}
\end{equation}
for some $\underline{z}$, $\bar{x} \in \R$.
Then the inequality
\begin{equation*}
    \left\langle Z,X\right\rangle \geq \bar{x}\sum_{k : \lambda_k(Z) <0}\lambda_k(Z) \geq n\bar{x}\min\{0,\underline{z}\}
\end{equation*}
holds.
\end{lemma}\label{lemma:lemmatoh}

\begin{proof}
 Let $Z= Q\Lambda Q^\top$ be an eigenvalue decomposition 
 of $Z$ 
 with $QQ^\top=I$ for some $Q \in \R^{n \times n}$ and $\Lambda=\Diag(\lambda(Z)).$
 Then
 \begin{align*}
      \left\langle Z,X\right\rangle & = \trace(Q\Lambda Q^\top X)  = \trace(\Lambda Q^\top X Q) \\
      & = \sum_{k=1}^n \lambda_k(Z) Q(:,k)^\top X Q(:,k).
 \end{align*}
 Because of~\eqref{eq:upplow}, we have $0 \leq Q(:,k)^\top X Q(:,k) \leq \bar{x}.$ Therefore
 $$\left\langle Z,X\right\rangle \geq \bar{x}\sum_{k : \lambda_k(Z) <0}\lambda_k(Z) \geq n\bar{x}\min\{0,\underline{z}\}.$$
\end{proof}

At this point we can present the following theorem of~\cite[Theorem 3.2]{JaChayKeil2007} adapted for DNNs.

\begin{theorem}\label{theorem:th32}
Consider the primal DNN~\eqref{eq:primaldnn}, let $X^*$ be an optimal solution and let $p^*$ be its optimal value. Given $y \in \mathbb{R}^m$ and $S \in \Sy$ with $S \geq 0$, set 
\begin{equation}\label{eq:Ztoh}
\ppp{Z} = C - \Aat y - S
\end{equation}
and suppose that
$\underline{z} \leq \lambda_{\min}(\ppp{Z}).$
Assume $\bar{x} \in \R$ such that $\bar{x} \geq \lambda_{\max}(X^*)$ is known.
Then the inequality
\begin{equation}\label{eq:validbound}
p^* \geq b^\top y + \bar{x}\sum_{k: \lambda_k(\ppp{Z})<0} \lambda_k(\ppp{Z}) \geq  b^\top y + n\bar{x}\min\{0, \underline{z}\}
\end{equation}
holds.
\end{theorem}
\begin{proof}
Let $X^*$ be optimal for the primal DNN~\eqref{eq:primaldnn}. Then
\begin{align*}
\left\langle C,X^*\right\rangle - b^\top y 
&= \left\langle C,X^*\right\rangle - \left\langle \Aa X^*, y\right\rangle = \left\langle C -\Aat y,X^*\right\rangle \\
&= \left\langle \ppp{Z}+ S,X^*\right\rangle = \left\langle \ppp{Z},X^*\right\rangle + \left\langle S,X^*\right\rangle.
\end{align*}
Since $S \geq 0$ and $X^* \geq 0$, the inequality
 $$\left\langle C,X^*\right\rangle \geq  b^\top y + \left\langle \ppp{Z},X^*\right\rangle$$
is satisfied and Lemma~\ref{lemma:lemmatoh} implies
$$p^* = \left\langle C,X^*\right\rangle \geq b^\top y + \left\langle \ppp{Z},X^*\right\rangle \geq  b^\top y + \bar{x}\sum_{k: \lambda_k(\ppp{Z}) <0} \lambda_k(\ppp{Z}) \geq  b^\top y + n \bar{x}\min\{0, \underline{z}\}, $$
which proves~\eqref{eq:validbound}.
\end{proof}

Theorem~\ref{theorem:th32} justifies to compute dual bounds via Algorithm~\ref{alg:errorbound}. 
If the matrix $\ppp{Z}$ defined in~\eqref{eq:Ztoh}
is positive semidefinite, then $(y,\ppp{Z},S)$, is a dual feasible solution and $b^\top y$
is already a bound. 
Otherwise, we decrease the dual objective function value $b^\top y$ of the infeasible point $(y,\ppp{Z},S)$ 
by adding the negative term $\bar{x}\sum\limits_{k: \lambda_k(\ppp{Z}) <0} \lambda_k(\ppp{Z})$ to it.  
In this way, we obtain a bound ($EB$ in Algorithm~\ref{alg:errorbound}) as proved by Theorem~\ref{theorem:th32}. 

Note that for the computation of the bound of Theorem~\ref{theorem:th32} it is not necessary to have a primal optimal solution $X^*$ at hand, only an upper bound on the maximum eigenvalue of an optimal solution is needed. Such an upper bound is known for example if there is an upper bound on the maximum eigenvalue of any feasible solution.

\begin{algorithm}[ht]
    \caption{Scheme for Computing Error Bounds}
    \label{alg:errorbound}
    \Input{$y \in \R^m$, $S \in \Sy$ with $S\geq 0$, 
    $\bar{x}\geq \lambda_{\max}(X^*)$}
    \begin{algorithmic}[1]
        \State $\ppp{Z} = C - \Aat y - S$ 
        \State Compute $\lambda(\ppp{Z})$
        \State $EB = b^\top y + \bar{x}\sum\limits_{k: \lambda_k(\ppp{Z})<0} \lambda_k(\ppp{Z}) $  
        \State \Return $EB$
    \end{algorithmic}
\end{algorithm}

\subsection{Dual Bounds through the Nightjet Procedure}
\label{sec:nightjet}
Next we will present a new procedure to obtain bounds. In 
contrast to the procedure described in the previous section, this approach will also provide a dual feasible 
solution.
The key ingredient to obtain such dual feasible solutions will be the following lemma.

\begin{lemma}\label{lem:NJProcedure}
We consider the primal DNN~\eqref{eq:primaldnn} and the dual DNN~\eqref{eq:dualdnn}.
Let $\ppp{Z} \in \sdp$. 
If 
\begin{equation}\label{eq:LP}
    \max_{y \in \R^m}\{ b^\top y \mid \Aat y \le C - \ppp{Z} \}
\end{equation}
has an optimal solution $\ppp{y}$, let $\ppp{S}= C -\ppp{Z} - \Aat \ppp{y}$. 
Then $(\ppp{y}, \ppp{S}, \ppp{Z})$ is a dual feasible solution. 
If~\eqref{eq:LP} is unbounded, then also~\eqref{eq:dualdnn} is unbounded.
If~\eqref{eq:LP} is infeasible, then there is no dual feasible solution with $\ppp{Z}$.
\end{lemma}
\begin{proof}
If~\eqref{eq:LP} has an optimal solution $\ppp{y}$, 
then it is easy to see that 
$\ppp{S} \geq 0$ by construction. 
Furthermore $\ppp{S} \in \Sy$ because $C$, $\ppp{Z}$, $\Aat y \in \Sy$.
Therefore $(\ppp{y}, \ppp{S}, \ppp{Z})$ is a dual feasible solution.
If~\eqref{eq:LP} is unbounded, 
then the same values of $y$ that make the objective function value of~\eqref{eq:LP} 
arbitrarily large can be used to make the objective function value of~\eqref{eq:dualdnn} 
arbitrary large, hence also~\eqref{eq:dualdnn} is unbounded.
Furthermore it is easy to see that~\eqref{eq:LP} is feasible 
if there is a dual feasible solution with $\ppp{Z}$. 
Hence if~\eqref{eq:LP} is infeasible, 
then there is no dual feasible solution with $\ppp{Z}$.
\end{proof}

Let $(y,S,Z,X)$ be any solution (not necessarily feasible)
to the primal DNN~\eqref{eq:primaldnn} and the dual DNN~\eqref{eq:dualdnn}.
In the back of our minds we think of them as the solutions we obtained by 
\texttt{ADAL+}, 
\texttt{DADAL+}, \texttt{ConicADMM3c} or
\texttt{DADMM3c}, so they are close to optimal solutions but not necessarily 
dual or primal feasible.
We want to obtain $\ppp{y}$, $\ppp{S}$ and $\ppp{Z}$ satisfying dual feasibility
\begin{equation}\label{eq:dualfeas}
    \Aat \ppp{y} +\ppp{Z}+\ppp{S}= C, \quad
    \ppp{Z} \in \sdp, \quad
    \ppp{S} \in \Sy, \quad 
    \ppp{S} \ge 0.
\end{equation}

We use Lemma~\ref{lem:NJProcedure} within the Nightjet procedure for obtaining such solutions in the following way.
From the given $Z$ we obtain the new positive semidefinite matrix $\ppp{Z}$ by projecting $Z$ onto the positive semidefinite cone. Then we solve the linear program~\eqref{eq:LP}.

If~\eqref{eq:LP} is infeasible, then we are neither able to construct a feasible 
dual solution nor to construct a dual bound. 
If~\eqref{eq:LP} is unbounded, then also the dual DNN~\eqref{eq:dualdnn} is unbounded and hence the 
primal DNN~\eqref{eq:primaldnn} is not feasible.
If~\eqref{eq:LP} has an optimal solution $\ppp{y}$, 
then we obtain a dual feasible solution 
$(\ppp{y}, \ppp{S}, \ppp{Z})$ with the help of Lemma~\ref{lem:NJProcedure}.
Furthermore the dual objective function value $b^\top \ppp{y}$ is a bound in this case, so we can return a dual feasible solution and a bound.
The Nightjet procedure is detailed in Algorithm~\ref{alg:NJHeuristicGeneral}.
\begin{algorithm}[ht]
    \caption{Scheme of the Nightjet Procedure}
    \label{alg:NJHeuristicGeneral}
    \Input{$Z\in \Sy$}
    \begin{algorithmic}[1]
        \State $\displaystyle \ppp{Z} = (Z)_+$
        \If{ $\{y \in \R^m \mid  \Aat y \le C - \ppp{Z} \} \not = \emptyset $ }
        \State $\displaystyle \ppp{y}= \argmax_{y \in \R^m} 
         \{b^\top y \mid \Aat y \le C - \ppp{Z} \}$  \label{yupdateNJHeuristic}        
        \Else 
        \State \Return ``No dual feasible solution and no bound found''
        \EndIf 
        \State $\ppp{S}= C -\ppp{Z} - \Aat \ppp{y}$ 
        \State $NB = b^\top \ppp{y}$
        \State \Return $NB$, $(\ppp{y}, \ppp{S}, \ppp{Z})$
    \end{algorithmic}
\end{algorithm}

To summarize, we have presented two different approaches to determine dual bounds 
for the primal DNN~\eqref{eq:primaldnn} from given $y$, $S$ and $Z$.

Note that the approaches are in the following sense complementary to each other: 
In the first approach from Jansson, Chaykin and Keil we fix $y$ and $S$ and obtain the bound from a 
newly computed $\ppp{Z}$, but we do not obtain a dual feasible solution.
In our second approach, the Nightjet procedure, we fix $\ppp{Z}$ to be the projection of $Z$ onto the positive semidefinite cone 
and then construct a feasible $\ppp{y}$ and $\ppp{S}$ from that.

Furthermore note that in the approach of Jansson, Chaykin and Keil the obtained bound is always less or 
equal to the dual objective function value of $y$, because a negative term is added to $b^\top y$, the dual objective function value using $y$. 
In contrast to that, it can happen in the Nightjet procedure that the bound is larger and hence 
better than $b^\top y$.
However, the Nightjet procedure comes with the drawback that it might be unable to produce a feasible solution. In this case one should continue running the ADMM to a higher precision and apply the procedure to the improved point.

\section{Numerical Experiments}\label{sec:numres}
In this section we present a comparison of the four ADMMs using the two procedures presented in Section~\ref{sec:SafeBounds} as post-processing phase. Towards that end 
we consider instances of one fundamental problem from combinatorial optimization, the stable set problem.

\subsection{The Stable Set Problem and an SDP Relaxation}
Given a graph $G$, let $V(G)$ be its set of vertices and $E(G)$ its set of edges. 
A subset of $V(G)$ is called stable, if no two vertices are adjacent.
The stability number $\alpha(G)$ is the largest possible cardinality of a stable set. It is NP-hard to compute the stability number~\cite{karpStableSetColoringNP} and it is even hard to approximate it~\cite{StableSetHardToApprox},
therefore upper bounds on the stability number are of interest.
One possible upper bound is the Lov\'asz theta function $\vartheta(G)$, see for example~\cite{matRel}.
The Lov\'asz theta function is defined as the optimal value of the 
SDP
\begin{equation}
\begin{aligned}
\vartheta(G) = 
 &\max  \, &&
 \left\langle J, X \right\rangle \\
 &\mbox{ s.t. } &&
 \begin{alignedat}[t]{2}
    \trace(X) &= 1\\
    X_{ij} &= 0 \quad \quad && \forall \{i,j\} \in E(G)
  \end{alignedat}\\
  &&&
  X\in \sdp,\nonumber 
\end{aligned}
\end{equation}
where $J$ is the $n$-by-$n$ matrix of all ones. 
Note that $\vartheta(G)$ -- as SDP of polynomial size -- can be computed to arbitrary precision in polynomial time.
Hence $\vartheta(G)$ is a polynomial computable upper bound on $\alpha(G)$.

Several attempts of improving $\vartheta(G)$ towards $\alpha(G)$ have been done.
One of the most recent ones is including the so called exact subgraph constraints into the SDP of computing $\vartheta(G)$, which make sure that for small subgraphs the solution is in the respective squared stable set  polytope~\cite{GaarRendl}.
This approach is a generalization of one of the first approaches to improve $\vartheta(G)$ in~\cite{SchrijverNonNeg}, which consisted of adding the constraint $X\geq 0$. Compared to $\vartheta(G)$ this leads to an even stronger bound on $\alpha(G)$ as the copositive cone is better approximated.
We denote by $\vartheta_+(G)$ the optimal objective function value of the DNN
\begin{equation}\label{theta+}
\begin{aligned}
\vartheta_+(G) = 
 &\max \, &&
 \left\langle J, X \right\rangle \\
 &\mbox{ s.t. } &&
 \begin{alignedat}[t]{2}
    \trace(X) &= 1\\
    X_{ij}   &= 0 \quad \quad &&\forall  \{i,j\} \in E(G)
  \end{alignedat}\\
  &&&
  X\in \sdp, \quad X\geq 0.  
\end{aligned}
\end{equation}
Note that in the DNN~\eqref{theta+} the matrix $\Aa \Aat$ is a diagonal matrix,
which leads to an inexpensive update of $y$ in the methods discussed.

\subsection{Dual Bounds for \texorpdfstring{$\vartheta_+(G)$}{Theta Plus} }
As already discussed in Section~\ref{sec:SafeBounds}, for a combinatorial optimization problem 
like the stable set problem, bounds on the objective function value are of huge importance.

The bound according to Jansson, Chaykin and Keil~\cite{JaChayKeil2007} can be used for computing
bounds on $\vartheta_+(G)$ very easily: 
We can set $\bar{x} = 1$, as for every feasible solution $X$ of~\eqref{theta+} 
we have $\trace(X) = 1$ and $X \in \sdp$ and hence $\lambda_{\max}(X) \leq 1$.

The computation of the dual bound with the Nightjet procedure simplifies 
drastically. In particular there is no need to solve the linear program~\eqref{eq:LP}, since the solution 
can be computed explicitly. 
To be more precise, the following holds.

\begin{lemma}
We consider the primal DNN~\eqref{theta+} to compute $\vartheta_+(G)$ and the dual of it.
Let $y_t$ be the dual variable for the 
constraint $\trace(X) = 1$ and $y_e$ be the dual variable for the constraint $X_{ij} = 0$ for every edge $e = \{i,j\} \in E(G)$.
Furthermore let $\ppp{Z} \in \sdp$ and let 
$
M = \max \left\{\ppp{Z}_{ij} \mid  \{i,j\}   \not\in E(G) \right\}.
$

If $M \ge 0 $, then it is not possible to construct a dual feasible solution with this $\ppp{Z}$.
If $-1 < M < 0$, then we can redefine $\ppp{Z}$ as
$
    \ppp{Z} = -\frac{1}{M}\ppp{Z},
$
and obtain a new $\ppp{Z}$ for which $M = - 1$.
If $M \leq -1$, then we obtain a dual feasible solution with
\begin{align*}
    \ppp{y}_t     &= \min \left\{- 1 - \ppp{Z}_{ii} \mid i \in \{1, 2, \dots, n\}\right\},\\
    \ppp{y}_e     &= 2(- 1 - \ppp{Z}_{ij}) \quad \quad \forall e = \{i,j\} \in E(G),\\
    \ppp{S} &= C -\ppp{Z} - \Aat \ppp{y}.
\end{align*}
\end{lemma}
\begin{proof}
We first consider the dual of~\eqref{theta+} in more detail.
To be consistent with 
our notation we replace the objective function $\max \,\left\langle J, X \right\rangle$ of~\eqref{theta+} 
with the equivalent objective function $-\min \,\left\langle -J, X \right\rangle$ in order to consider a 
primal minimization problem as in the primal DNN~\eqref{eq:primaldnn}. We introduce one dual variable $y_t$ for the 
constraint $\trace(X) = 1$ and one dual variable $y_e$ for the constraint $X_{ij} = 0$ for every edge $e = \{i,j\} \in E(G)$. Then the 
dual of~\eqref{theta+} is given as
\begin{equation}\label{eq:dtheta+}
\begin{aligned}
 -&\max \,&&
 y_t \\
 &\mbox{ s.t. } &&
 \begin{alignedat}[t]{2}
                y_t + Z_{ii} + S_{ii} &= -1 \quad \quad &&\forall  i  \in \{1, 2, \dots, n\} \\
    \tfrac{1}{2}y_e + Z_{ij} + S_{ij} &= -1 \quad       &&\forall  e = \{i,j\} \in E(G) \\
                      Z_{ij} + S_{ij} &= -1 \quad       &&\forall   \{i,j\} \not \in E(G)
  \end{alignedat}\\
  &&& 
  Z\in \sdp, \quad S\in \Sy, \quad S\geq 0, \quad y_t \in \R, \quad y_e \in \R \quad \forall e \in E(G). 
\end{aligned}
\end{equation}

Now we apply Lemma~\ref{lem:NJProcedure} for~\eqref{theta+}.
Thus we replace the dual variable $Z$ with the fixed $\ppp{Z} \in \sdp$ and 
the linear program~\eqref{eq:LP} becomes
\begin{equation}\label{eq:lpNJtheta+}
    \begin{aligned}
    -&\max \, && 
     y_t\\
     &\mbox{ s.t. }  && 
     \begin{alignedat}[t]{2}
        y_t             &\le -1 -  \ppp{Z}_{ii}\quad\quad && \forall  i  = \{1, 2, \dots, n\} \\
        \tfrac{1}{2}y_e  &\le -1 - \ppp{Z}_{ij} && \forall  e = \{i,j\} \in E(G) \\
          \ppp{Z}_{ij}  &\le -1  && \forall   \{i,j\} \not \in E(G)
      \end{alignedat}\\
      &&&
      y_t \in \R, \quad y_e \in \R \quad \forall e \in E(G).
    \end{aligned}
\end{equation}
Clearly this linear program is bounded and 
detecting infeasibility or constructing an optimal solution is straightforward. Indeed, let
$
    M = \max \left\{\ppp{Z}_{ij}\mid \{i,j\}   \not\in E(G) \right\},
$
then it is easy to see that~\eqref{eq:lpNJtheta+} is infeasible if $M > - 1$. However, if $-1 < M < 0$ holds, then we can redefine $\ppp{Z}$ as
$
    \ppp{Z} = -\frac{1}{M}\ppp{Z},
$
and obtain a new $\ppp{Z}$ for which $M=-1$. On the contrary, if $ M \ge 0$, we can not update $\ppp{Z}$ in a straightforward way.
If $M \le -1$, then~\eqref{eq:lpNJtheta+} is feasible and we can construct the optimal solution as
\begin{align*}
    y_t & = \min \left\{- 1 - \ppp{Z}_{ii}\mid i \in \{1, 2, \dots, n\} \right\},\\
    y_e & = 2(- 1 - \ppp{Z}_{ij}) \quad \quad \forall e = \{i,j\} \in E(G).
\end{align*}
Then we let $\ppp{S} = C -\ppp{Z} - \Aat \ppp{y}$ and due to Lemma~\ref{lem:NJProcedure} this yields a feasible dual solution $(\ppp{y}, \ppp{S}, \ppp{Z})$.
\end{proof}

Hence, for computing a dual bound for $\vartheta_+(G)$ it is not necessary to solve the linear program~\eqref{eq:LP}, 
but the solution of it can be written down explicitly. 
This explicit solution is used by the Nightjet procedure for $\vartheta_+(G)$ to obtain $\ppp{y}$. 
The computation of $\ppp{Z}$ and $\ppp{S}$ is the same as in the original Nighjet procedure. 
The pseudocode of the Nightjet procedure applied to the computation of $\vartheta_+(G)$ 
can be found in Algorithm~\ref{alg:NJHeuristicSpecific}.

\begin{algorithm}[ht]
    \caption{Scheme of the Nightjet Procedure for $\vartheta_+(G)$}
    \label{alg:NJHeuristicSpecific}
    \Input{$Z\in \Sy$}
    \begin{algorithmic}[1]
        \State $\ppp{Z} = (Z)_+$  
        \State $M = \max \left\{\ppp{Z}_{ij}\mid \{i,j\}   \not\in E(G)\right\}$
        \If{$M \ge 0$}
        \State \Return ``No dual feasible solution and no bound found''
        \ElsIf{$0 > M > -1$}
        \State $\ppp{Z} = \frac{1}{-M}\ppp{Z}$
        \EndIf
        \State $\ppp{y}_t  = \min \left\{- 1 - \ppp{Z}_{ii}\mid i \in \{1, 2, \dots, n\}\right\}$
        \State $\ppp{y}_e  = 2 (- 1 - \ppp{Z}_{ij}) \quad \forall e = \{i,j\} \in E(G)$
        \State $\ppp{S}= C -\ppp{Z} - \Aat \ppp{y}$ 
        \State $NB = b^\top \ppp{y}$
        \State \Return $NB$, $(\ppp{y}, \ppp{S}, \ppp{Z})$
    \end{algorithmic}
\end{algorithm}

\subsection{Comparison of the Evolution of the Dual Bounds}
In the following, we give a numerical comparison of the two procedures for the computation of bounds 
for $\vartheta_+(G)$ on one instance from the second DIMACS implementation challenge~\cite{Johnson:1996}, namely 
\texttt{johnson8\_2\_4}.
For this instance the stability number $\alpha(G)$ and $\vartheta_+(G)$ coincide 
and both are equal to $4$.

In Figure~\ref{fig:EvolBounds}, we show the evolution of the bounds along the iterations for \texttt{ADAL+}, \texttt{DADAL+}, \texttt{ConicADMM3c} and \texttt{DADMM3c}.
For each algorithm we report the dual objective function value (dualOfv), the bound computed according to
Jansson, Chaykin and Keil~\cite{JaChayKeil2007} ($EB$) and the bound computed by 
the Nightjet procedure ($NB$) at every iteration. 

Note that in some iterations the dual objective function value is not a bound on $\vartheta_+(G) = 4$ and hence also not on $\alpha(G)$.
This is due to the fact that the solution considered is not dual feasible. (The criteria are satisfied only to moderate precision.)

We observe that for \texttt{ADAL+}, \texttt{DADAL+} and \texttt{ConicADMM3c} the Nightjet bound is always less or equal than the error bound and in several 
iterations it is significantly better, in particular at the iterations in the beginning. 
Hence our Nightjet procedure is an effective tool to obtain dual bounds.
Note that every ADMM keeps $Z$ positive semidefinite along the iterations 
(see Table~\ref{tab:opt_cond_true}) and 
this may be in favor of the Nightjet procedure.

\begin{figure}[ht]
  \begin{center}
   \includegraphics[trim={1.5cm 0.5cm 1cm 0},clip,width=0.48\textwidth]{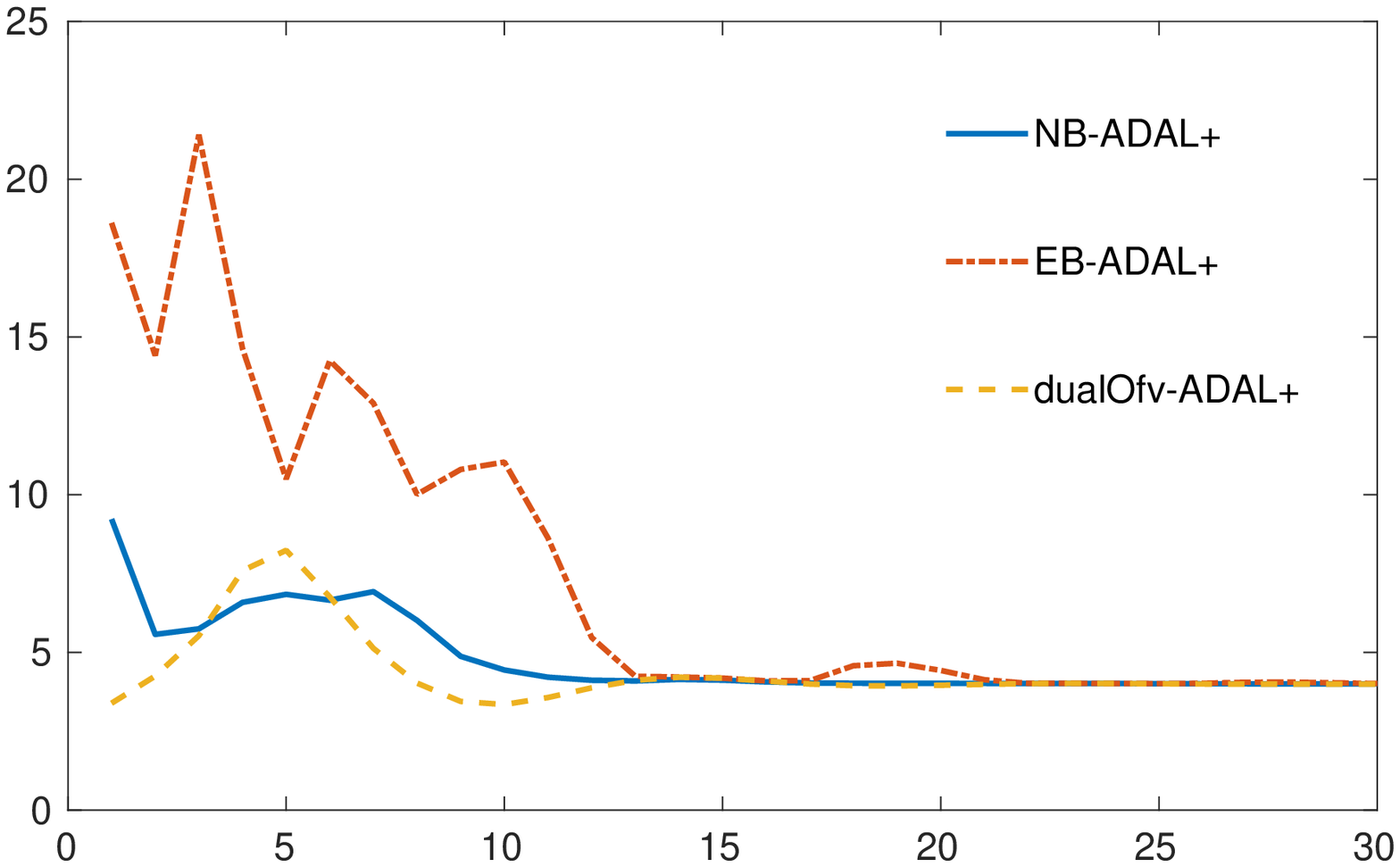}
   \includegraphics[trim={1.5cm 0.5cm 1cm 0},clip,width=0.48\textwidth]{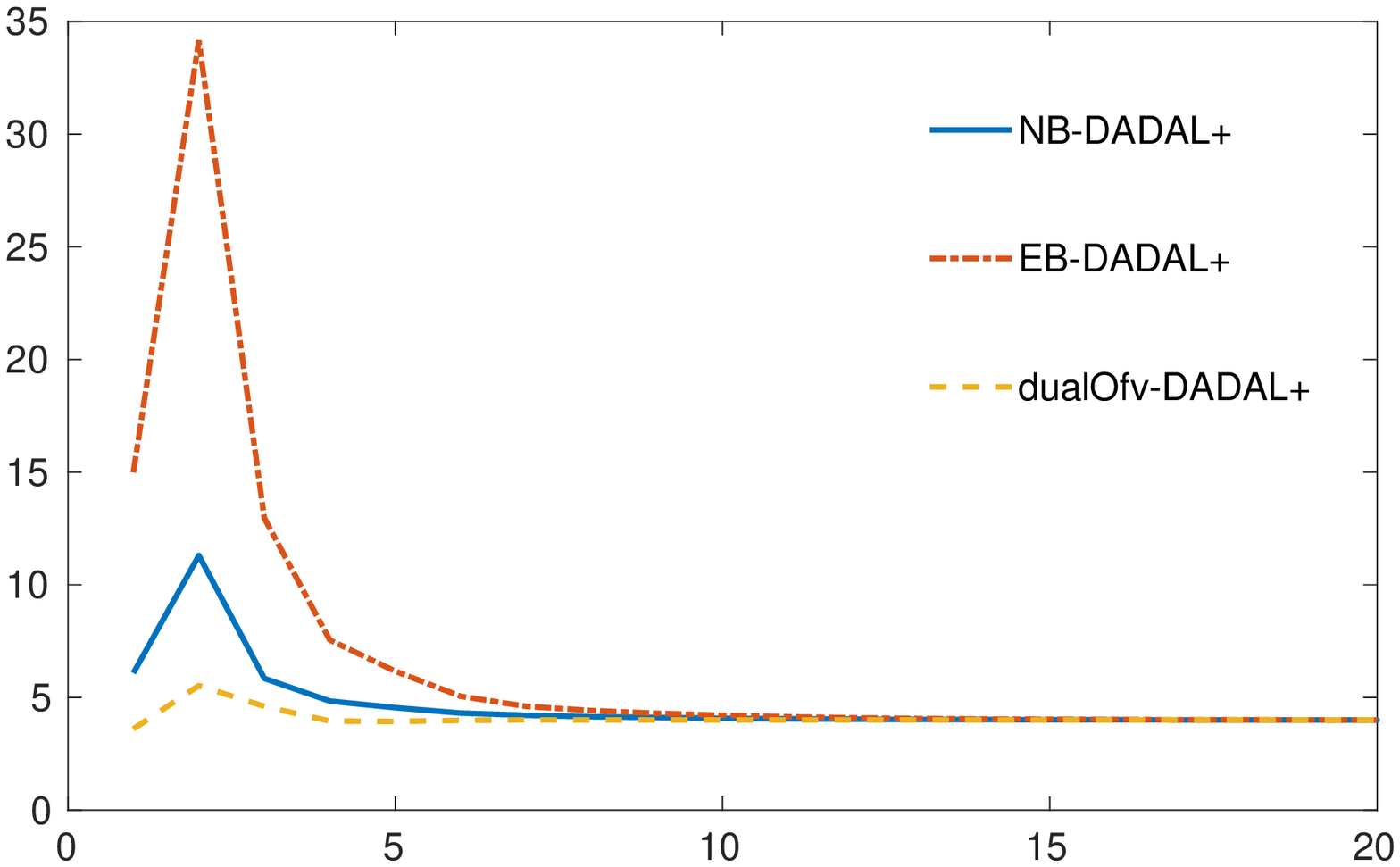}
   \\
   \includegraphics[trim={1.5cm 0.5cm 1cm 0},clip,width=0.48\textwidth]{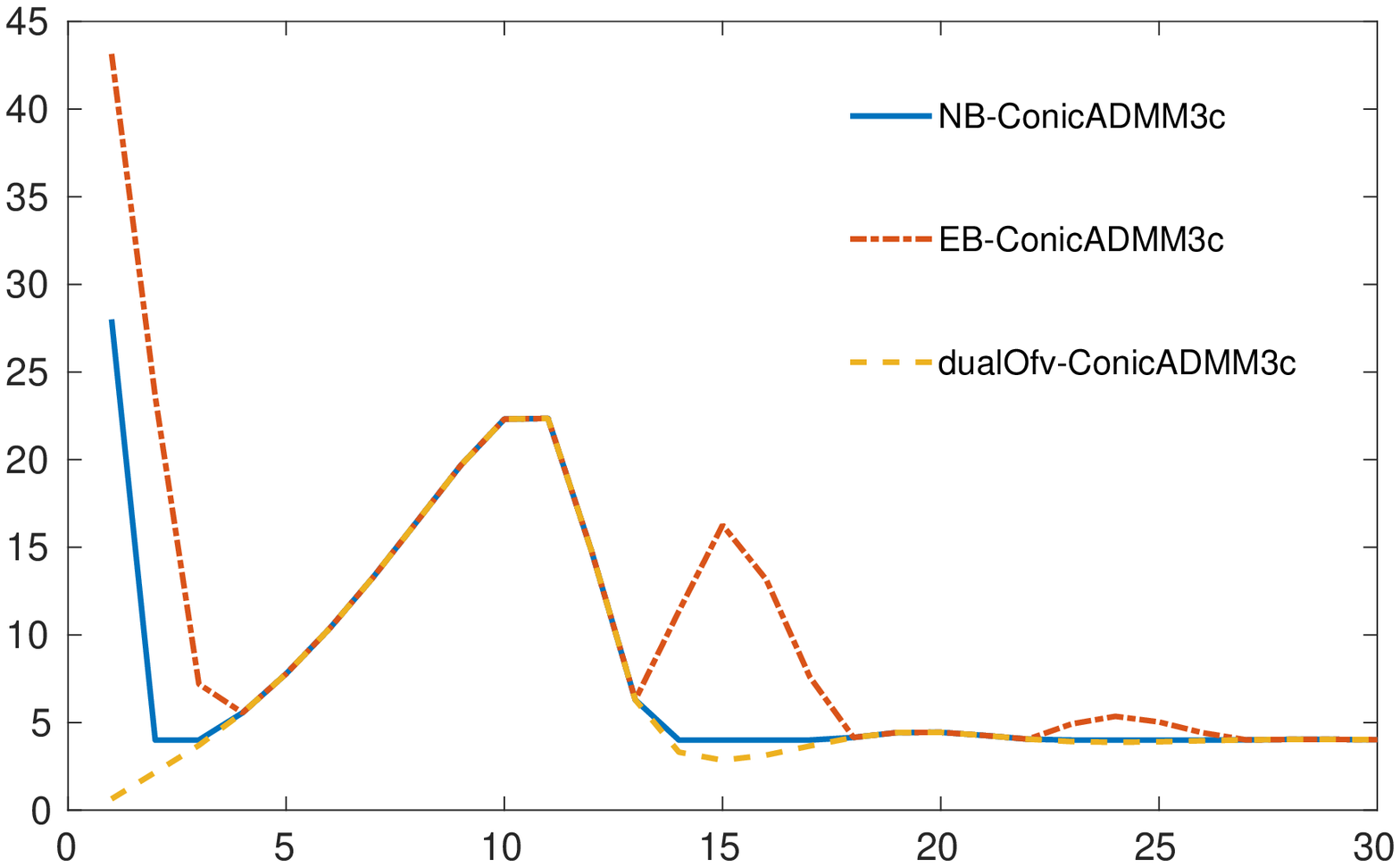}
   \includegraphics[trim={1.5cm 0.5cm 1cm 0},clip,width=0.48\textwidth]{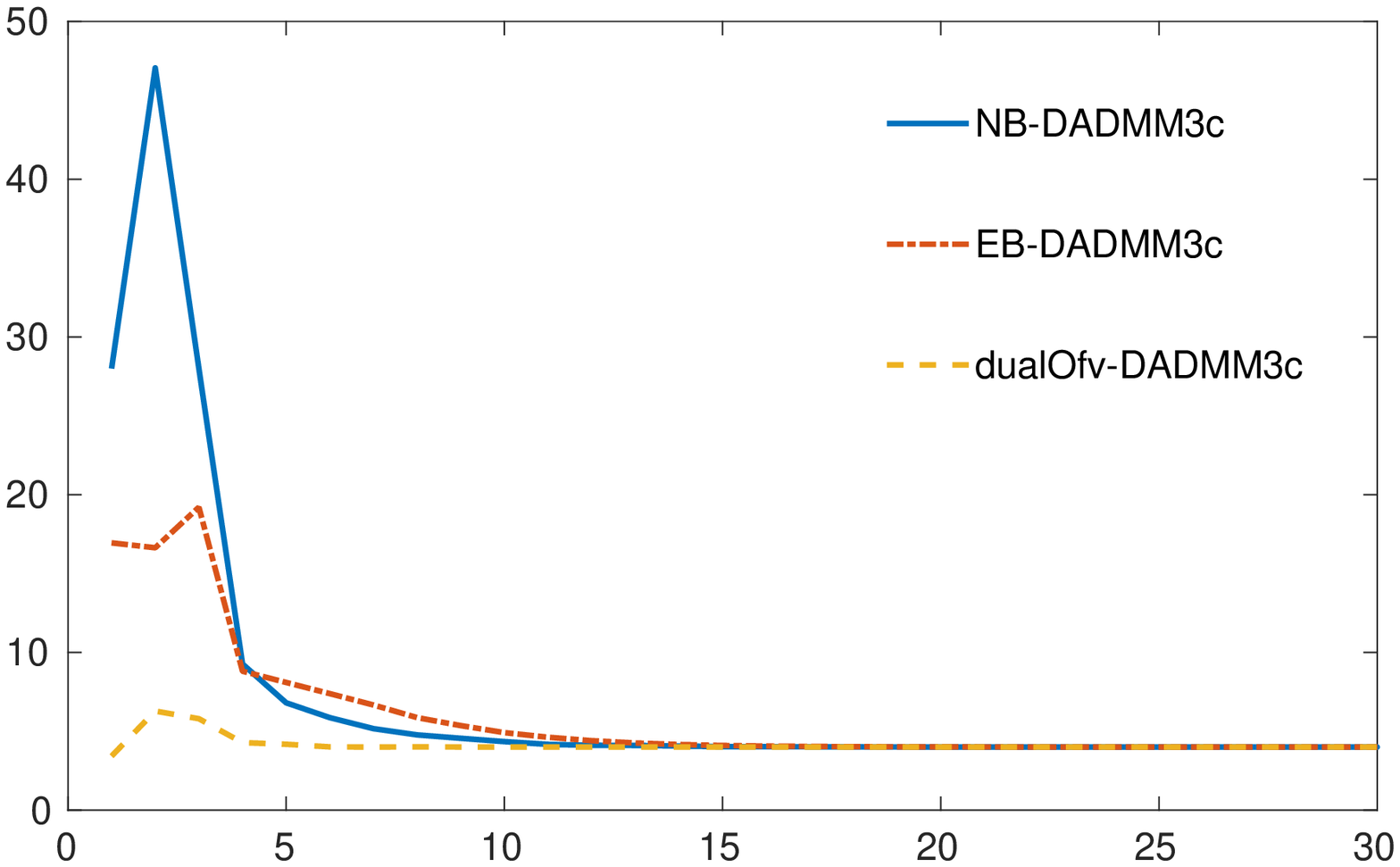}
  \end{center}
  \caption{Evolution of the computed bounds on the instance \texttt{johnson8\_2\_4}.}
  \label{fig:EvolBounds}
\end{figure}

\subsection{Computational Setup}
In our numerical experiments we compare the performance 
of \texttt{ADAL+}, \texttt{DADAL+}, \texttt{ConicADMM3c} and
\texttt{DADMM3c} on $66$ instances of the DNN~\eqref{theta+} to compute $\vartheta_+(G)$. The graphs are taken
from the second DIMACS implementation challenge~\cite{Johnson:1996}.
Note that in that challenge the task was to find a maximum clique of several graphs, 
so we consider the complement graphs of the graphs in~\cite{Johnson:1996}.
In Table~\ref{tab:instances}, 
for each instance on a graph $G$, 
we report its name (Problem) and its dimension 
(the number of vertices $n$ and the number of edges $m$ of $G$).
The value of the Lov\'asz theta function $\vartheta(G)$ for many of these instances can be found in \cite{ThetaValues_GiLeRoSm:13} and \cite{MaPoReWi:09}, in this article we exclusively focus on $\vartheta_+(G)$.

\begin{table}[ht]
\caption{Data of the DIMACS instances considered in~\cite{Johnson:1996}.}\label{tab:instances}
\begin{center}
\small{
\renewcommand{\arraystretch}{1.05}
\begin{tabular}{|l r r |l r r|l r r|}
    \hline
    Problem & $n$ & $m$ & Problem & $n$ & $m$ & Problem & $n$ & $m$ \\ 
\hline
johnson8\_2\_4    &  28 &       168   &    
hamming8\_4     & 256 &     11776   &  
p\_hat500\_2    & 500 &     61804  \\ 
MANN\_a9     &  45 &       72    &	   
p\_hat300\_1    & 300 &     33917   &  
p\_hat500\_3    & 500 &     30950  \\ 
hamming6\_2     &  64 &      192    &	   
p\_hat300\_2    & 300 &     22922   &  
p\_hat700\_1    & 700 &    183651  \\ 
hamming6\_4     &  64 &     1312    &	   
p\_hat300\_3    & 300 &     11460   &  
p\_hat700\_2    & 700 &    122922  \\ 
johnson8\_4\_4    &  70 &       560   &	     
MANN\_a27     & 378 &       702   &  
p\_hat700\_3    & 700 &     61640  \\ 
johnson16\_2\_4   & 120 &     1680    &	   
brock400\_1     & 400 &     20077   &     
keller5      & 776 &     74710  \\ 
keller4      & 171 &    5100     &	   
brock400\_2     & 400 &     20014   &  
brock800\_1     & 800 &   112095   \\ 
brock200\_1     & 200 &     5066    &	   
brock400\_3     & 400 &     20119   &  
brock800\_2     & 800 &   111434   \\ 
brock200\_2     & 200 &    10024    &	   
brock400\_4     & 400 &     20035   &  
brock800\_3     & 800 &   112267   \\ 
brock200\_3     & 200 &     7852    &	 
san400\_0\_5\_1   & 400 &     39900   &  
brock800\_4     & 800 &   111957   \\ 
brock200\_4     & 200 &     6811    &	 
san400\_0\_7\_1   & 400 &     23940   & 
p\_hat1000\_1    & 1000&    377247  \\ 
c\_fat200\_1    & 200 &     18366   &	 
san400\_0\_7\_2   & 400 &     23940   & 
p\_hat1000\_2    & 1000&    254701  \\ 
c\_fat200\_2    & 200 &     16665   &	 
san400\_0\_7\_3   & 400 &     23940   & 
p\_hat1000\_3    & 1000&    127754  \\ 
c\_fat200\_5    & 200 &     11427   &	 
san400\_0\_9\_1   & 400 &      7980   &     
san1000      & 1000&   249000   \\ 
san200\_0\_7\_1   & 200 &      5970   &	  
sanr400\_0\_5    & 400 &     39816   & 
hamming10\_2     & 1024&      5120  \\ 
san200\_0\_7\_2   & 200 &     5970    &	  
sanr400\_0\_7    & 400 &     23931   & 
hamming10\_4     & 1024&     89600  \\ 
san200\_0\_9\_1   & 200 &     1990    &	 
johnson32\_2\_4   & 496 &     14880   &    
MANN\_a45     & 1035&      1980  \\ 
san200\_0\_9\_2   & 200 &     1990    &	  
c\_fat500\_10    & 500 &     78123   & 
p\_hat1500\_1    & 1500&    839327  \\ 
san200\_0\_9\_3   & 200 &     1990    &	   
c\_fat500\_1    & 500 &    120291   & 
p\_hat1500\_2    & 1500&    555290  \\ 
san200\_0\_7    & 200 &      6032   &	   
c\_fat500\_2    & 500 &    115611   & 
p\_hat1500\_3    & 1500&    277006  \\ 
san200\_0\_9    & 200 &      2037   &	   
c\_fat500\_5    & 500 &    101559   &    
MANN\_a81     & 3321&      6480  \\ 
hamming8\_2     & 256 &      1024   &	   
p\_hat500\_1    & 500 &     93181   &     
keller6      & 3361&   1026582  \\
\hline
\end{tabular}
}
\end{center}
\end{table}

We implemented the four algorithms detailed in Sections~\ref{sec:admm} 
and~\ref{sec:Zfactorization} in MATLAB~R2019a.
In all computations, we set the accuracy~level~$\varepsilon$ to $10^{-5}$ and we set a time limit of $3600$ seconds CPU time.
In both \texttt{DADAL+} and \texttt{DADMM3c} we perform two iterations of Algorithm~\ref{alg:OneItUncMax} 
in order to update $(y,V)$. 

It is known that the performance of ADMMs strongly depends on the update of the penalty parameter $\sigma$.
In all implementations, we use the strategy described by Lorenz
and Tran-Dinh~\cite{Lorenz2019}, so in iteration $k$ we set
$$\sigma^k = \frac{\lVert X^k \rVert}{\lVert Z^k \rVert}.$$

The experiments were carried out on an Intel Core i7 processor running at 3.1~GHz under Linux.

\subsection{Comparison between \texttt{ADAL+} and \texttt{DADAL+}}
In Table~\ref{tab:ADAL+} we report the results obtained with \texttt{ADAL+} and \texttt{DADAL+}
on the $66$ instances of computing $\vartheta_+(G)$ detailed in Table~\ref{tab:instances}.
We include the following data for the
comparison: For each instance, we report its name (Problem) and its stability number ($\alpha$) and
for each of the two algorithms, we report the dual objective function value obtained (d ofv),  the bound obtained by computing the error 
bound described in Section~\ref{sec:errorBound}
($EB$), 
the bound obtained by applying the Nightjet procedure described in Section~\ref{sec:nightjet} ($NB$), the number of iterations (it) and the CPU time needed to 
satisfy the stopping criterion (time).

As a further comparison, we report in Figure~\ref{fig:PP-Adal} 
the performance profiles of \texttt{ADAL+} and \texttt{DADAL+}
with respect to the number of iterations and the CPU time. 
These performance profiles are obtained in the following way. 
Given our set of solvers~$\mathcal{S}$ and a set of problems $\mathcal{P}$, 
we compare the performance of  a solver $s \in \mathcal{S}$ on problem $p \in \mathcal{P}$
against the best performance obtained by any solver in~$\mathcal{S}$ 
on the same problem. To this end we define the performance ratio
$
r_{p,s} = t_{p,s}/\min\{t_{p,s^\prime} \mid s^\prime \in\mathcal{S}\},
$
where $t_{p,s}$ is the measure we want to compare, and we consider a cumulative distribution
function~$\rho_s(\tau) = |\{p\in \mathcal{P} \mid r_{p,s}\leq \tau \}| /|\mathcal{P}|$.
The performance profile for $s \in S$ is the plot of the function $\rho_s$.

Note that both \texttt{ADAL+} and \texttt{DADAL+} stopped on $7$ instances because of the time limit.
In the performance profiles, we exclude those instances where at least one of the solvers 
exceeds the time limit.

It is clear from the results on Table~\ref{tab:ADAL+} and from the performance profiles that \texttt{DADAL+} 
performs much less iterations than \texttt{ADAL+}. 
However, this does not always correspond to an improvement in terms of computational time as the double update of $y$ is an expensive operation.

With respect to the CPU time, Figure~\ref{fig:PP-Adal} shows that the performance of the two algorithms is similar, even if
\texttt{DADAL+} slightly outperforms \texttt{ADAL+} as its curve is always above the other one.

If we consider the dual objective function value in Table~\ref{tab:ADAL+} we see that in fact the dual objective function value obtained by \texttt{ADAL+} and \texttt{DADAL+} is often not a bound, for example on the instances \texttt{hamming6\_4}, \texttt{c\_fat200\_1}, \texttt{san200\_0\_7\_1}, \texttt{san400\_0\_9\_1}, \texttt{c\_fat500\_1} and \texttt{c\_fat500\_5}. This shows that a procedure for obtaining a bound from the approximate solution is indeed of major importance. 

Regarding the quality of the bounds, the Nightjet procedure is able to obtain better bounds with respect 
to the error bounds, both when applied as post-processing phase for \texttt{ADAL+}
and for \texttt{DADAL+}, for the vast majority of the instances. 
The improvement is particularly impressive when looking at those instances where the time limit is exceeded.
We want to further highlight that the bound obtained from the Nightjet procedure 
comes from a newly computed feasible dual solution. This means that applying the Nightjet procedure as post-processing does not only 
guarantee a bound generally better than the one obtained by the error bounds,
but it also provides a dual feasible solution.

\begin{center}
\ifFourOR 
\small{
\else 
\scriptsize{
\fi 
\setlength{\tabcolsep}{4.5pt}
\renewcommand{\arraystretch}{1.05}
\begin{longtable}{|l|rrrrr|rrrrr|r|}
\caption{Comparison between \texttt{ADAL+} and \texttt{DADAL+} on DIMACS instances~\cite{Johnson:1996}.}\label{tab:ADAL+}\\
\hline
\footnotesize
& \multicolumn{5}{c|}{ \texttt{ADAL+}} & \multicolumn{5}{c|}{ \texttt{DADAL+}} &\\
 Problem & d ofv & $EB$ & $NB$ & it & time &  d ofv  & $EB$ & $NB$ & it & time & $\alpha$ \\ 
 \endfirsthead
 \multicolumn{12}{c}{{\tablename\ \thetable{} -- continued from previous page}} \\\hline
& \multicolumn{5}{c|}{ \texttt{ADAL+}} & \multicolumn{5}{c|}{ \texttt{DADAL+}} & \\
 Problem & d ofv & $EB$ & $NB$ & it & time &  d ofv & $EB$ & $NB$ & it & time & $\alpha$ \\ 
 \hline
\hline
 \endhead
 \hline \multicolumn{12}{r}{{continued on next page}} \\
 \endfoot
 \hline
 \endlastfoot
\hline
\hline
johnson8\_2\_4    &  3.99999 &  4.00037 &  4.00012 &    44 &     1.1   &   4.00000 &  4.00025 &  4.00009 &    25  &    0.5 &   4   \\	   
     MANN\_a9     &  17.4750 &  17.4756 &  17.4755 &   765 &     1.1   &   17.4750 &  17.4756 &  17.4755 &   510  &    0.9 &  16   \\	   
  hamming6\_2     &  32.0005 &  32.0005 &  32.0004 &   669 &     1.9   &   31.9996 &  32.0004 &  32.0000 &   250  &    1.0 &  32   \\	   
  hamming6\_4     &  3.99994 &  4.00197 &  4.00016 &    56 &     0.1   &   3.99998 &  4.00110 &  4.00010 &    26  &    0.1 &   4   \\	   
johnson8\_4\_4    &  13.9998 &  14.0016 &  14.0002 &   135 &     0.3   &   14.0001 &  14.0001 &  14.0004 &    47  &    0.3 &  14   \\	   
johnson16\_2\_4   &  8.00000 &  8.00166 &  8.00034 &    89 &     0.6   &   8.00000 &  8.00411 &  8.00037 &    35  &    0.3 &   8   \\	   
     keller4      &  13.4659 &  13.4701 &  13.4667 &   764 &     8.7   &   13.4659 &  13.4716 &  13.4669 &   260  &    5.2 &  11   \\	   
  brock200\_1     &  27.1968 &  27.2003 &  27.1978 &   312 &     4.9   &   27.1966 &  27.2002 &  27.2007 &   222  &    7.2 &  21   \\	   
  brock200\_2     &  14.1310 &  14.1367 &  14.1325 &   158 &     2.6   &   14.1310 &  14.1359 &  14.1335 &   150  &    5.0 &  12   \\	   
  brock200\_3     &  18.6718 &  18.6764 &  18.6727 &   219 &     3.5   &   18.6718 &  18.6764 &  18.6745 &   146  &    3.8 &  15   \\	   
  brock200\_4     &  21.1211 &  21.1253 &  21.1220 &   264 &     4.2   &   21.1210 &  21.1254 &  21.1246 &   185  &    4.8 &  17   \\	   
  c\_fat200\_1    &  11.9999 &  12.0008 &  12.0006 &   339 &     5.2   &   11.9999 &  12.0009 &  12.0002 &   133  &    3.7 &  12   \\	   
  c\_fat200\_2    &  23.9981 &  24.0000 &  24.0000 &  2686 &    44.3   &   24.0015 &  24.0015 &  24.0014 &   782  &   21.9 &  24   \\	   
  c\_fat200\_5    &  60.3443 &  60.3483 &  60.3456 &  1218 &    19.1   &   60.3452 &  60.3474 &  60.3465 &   681  &   18.1 &  58   \\	   
san200\_0\_7\_1   &  29.9980 &  30.0000 &  30.0000 &  3441 &    54.6   &   29.9986 &  30.0000 &  30.0000 &   809  &   21.6 &  30   \\	   
san200\_0\_7\_2   &  18.0019 &  18.0019 &  18.0019 &  8782 &   144.6   &   18.0014 &  18.0014 &  18.0015 &  7614  &  202.6 &  18   \\	   
san200\_0\_9\_1   &  69.9980 &  70.0000 &  70.0000 &  3668 &    56.3   &   70.0009 &  70.0009 &  70.0008 &   670  &   16.1 &  70   \\	   
san200\_0\_9\_2   &  60.0020 &  60.0020 &  60.0019 &  3551 &    53.9   &   59.9991 &  60.0000 &  60.0000 &   688  &   17.0 &  60   \\	   
san200\_0\_9\_3   &  44.0016 &  44.0016 &  44.0016 & 11792 &   190.9   &   44.0010 &  44.0010 &  44.0014 & 10386  &  263.0 &  44   \\	   
  san200\_0\_7    &  23.6333 &  23.6372 &  23.6344 &   313 &     4.8   &   23.6332 &  23.6370 &  23.6364 &   218  &    5.9 &  18   \\	   
  san200\_0\_9    &  48.9046 &  48.9077 &  48.9063 &   403 &     6.1   &   48.9044 &  48.9070 &  48.9083 &   400  &   10.0 &  42   \\	   
  hamming8\_2     &  128.002 &  128.002 &  128.002 &  2760 &    83.8   &   128.001 &  128.001 &  128.001 &   694  &   29.6 & 128   \\	   
  hamming8\_4     &  16.0002 &  16.0002 &  16.0012 &   121 &     3.6   &   16.0001 &  16.0001 &  16.0011 &    52  &    3.0 &  16   \\	   
  p\_hat300\_1    &  10.0203 &  10.0348 &  10.0232 &   380 &    15.6   &   10.0202 &  10.0372 &  10.0208 &   457  &   30.3 &   8   \\	   
  p\_hat300\_2    &  26.7143 &  26.7154 &  26.7153 &  2315 &    94.2   &   26.7138 &  26.7274 &  26.7157 &  3576  &  230.0 &  25   \\	   
  p\_hat300\_3    &  40.7008 &  40.7096 &  40.7030 &   604 &    24.6   &   40.7003 &  40.7075 &  40.7061 &   817  &   51.1 &  36   \\	   
    MANN\_a27     &  132.762 &  132.765 &  132.763 &  2037 &   136.4   &   132.762 &  132.766 &  132.765 &   768  &   77.9 & 126   \\	   
  brock400\_1     &  39.3307 &  39.3438 &  39.3377 &   215 &    16.7   &   39.3308 &  39.3433 &  39.3391 &   148  &   18.5 &  27   \\	   
  brock400\_2     &  39.1963 &  39.2083 &  39.2024 &   216 &    16.7   &   39.1964 &  39.2080 &  39.2037 &   152  &   19.0 &  29   \\	   
  brock400\_3     &  39.1602 &  39.1742 &  39.1673 &   211 &    16.5   &   39.1603 &  39.1724 &  39.1679 &   149  &   19.2 &  31   \\	   
  brock400\_4     &  39.2313 &  39.2455 &  39.2361 &   208 &    15.6   &   39.2313 &  39.2440 &  39.2363 &   146  &   18.3 &  33   \\	   
san400\_0\_5\_1   &  13.0038 &  13.0038 &  13.0036 &  8135 &   629.4   &   13.0034 &  13.0034 &  13.0032 &  4705  &  528.4 &  13   \\	   
san400\_0\_7\_1   &  39.9962 &  40.0000 &  40.0000 &  6654 &   516.5   &   40.0038 &  40.0038 &  40.0037 &  1723  &  213.6 &  40   \\	   
san400\_0\_7\_2   &  30.0037 &  30.0037 &  30.0036 &  7999 &   623.3   &   30.0032 &  30.0032 &  30.0031 &  3863  &  468.0 &  30   \\	   
san400\_0\_7\_3   &  22.0000 &  22.0084 &  22.0044 &   601 &    47.9   &   22.0002 &  22.0143 &  22.0010 &   265  &   31.1 &  22   \\	   
san400\_0\_9\_1   &  99.9960 &  100.000 &  100.000 &  7212 &   547.5   &   99.9978 &  100.000 &  100.000 &  1499  &  167.8 & 100   \\	   
 sanr400\_0\_5    &  20.1782 &  20.1924 &  20.1794 &   164 &    12.6   &   20.1782 &  20.1924 &  20.1839 &   115  &   15.2 &  13   \\	   
 sanr400\_0\_7    &  33.9666 &  33.9794 &  33.9727 &   186 &    14.5   &   33.9666 &  33.9790 &  33.9724 &   141  &   17.3 &  21   \\	   
johnson32\_2\_4   & 15.9999  &  16.0047 &  16.0000 &   272 &    36.8   &   16.0000 &  16.0264 &  16.0005 &    69  &   12.4 &  16   \\	   
 c\_fat500\_10    &  126.003 &  126.003 &  126.003 &  3752 &   509.8   &   126.000 &  126.001 &  126.001 &  2772  &  549.4 & 126   \\	   
  c\_fat500\_1    &  13.9992 &  14.0113 &  14.0015 &   488 &    64.7   &   13.9987 &  14.0147 &  14.0004 &   251  &   58.9 &  14   \\	   
  c\_fat500\_2    &  25.9988 &  26.0136 &  26.0007 &   554 &    74.7   &   25.9990 &  26.0116 &  26.0003 &   302  &   69.8 &  26   \\	   
  c\_fat500\_5    &  63.9970 &  64.0040 &  64.0008 &  2740 &   374.1   &   63.9969 &  64.0044 &  64.0007 &  1104  &  252.4 &  64   \\	   
  p\_hat500\_1    &  13.0080 &  13.0401 &  13.0102 &   342 &    46.4   &   13.0079 &  13.0454 &  13.0088 &   380  &   81.0 &   9   \\	   
  p\_hat500\_2    &  38.5606 &  38.5638 &  38.5653 &  2159 &   290.5   &   38.5594 &  38.5871 &  38.5619 &  4404  &  907.9 &  36   \\	   
  p\_hat500\_3    &  57.8122 &  57.8287 &  57.8220 &   850 &   117.5   &   57.8111 &  57.8251 &  57.8155 &  1021  &  208.3 &  $\geq 50$     \\ 
  p\_hat700\_1    &  15.0452 &  15.0996 &  15.0500 &   361 &   117.4   &   15.0451 &  15.1077 &  15.0460 &   422  &  204.8 &  11   \\	   
  p\_hat700\_2    &  48.4420 &  48.4466 &  48.4463 &  2168 &   696.3   &   48.4401 &  48.4827 &  48.4436 &  4901  & 2388.1 &  44   \\	   
  p\_hat700\_3    &  71.7569 &  71.7761 &  71.7701 &  1234 &   391.6   &   71.7551 &  71.7853 &  71.7736 &  1543  &  727.0 &  62   \\	   
     keller5      &  30.9956 &  31.0461 &  30.9987 &  1684 &   767.0   &   30.9956 &  31.0469 &  30.9984 &  1970  & 1136.6 &  27   \\	   
  brock800\_1     &  41.8673 &  41.9080 &  41.8712 &   232 &   118.6   &   41.8673 &  41.9107 &  41.8701 &   107  &   81.8 &  23   \\	   
  brock800\_2     &  42.1043 &  42.1446 &  42.1071 &   231 &   121.6   &   42.1043 &  42.1477 &  42.1067 &   107  &   80.1 &  24   \\	   
  brock800\_3     &  41.8825 &  41.9235 &  41.8860 &   234 &   125.4   &   41.8825 &  41.9251 &  41.8859 &   109  &   84.9 &  25   \\	   
  brock800\_4     &  42.0006 &  42.0426 &  42.0037 &   232 &   122.9   &   42.0006 &  42.0429 &  42.0034 &   108  &   84.1 &  26   \\	   
 p\_hat1000\_1    &  17.5225 &  17.5888 &  17.5303 &   475 &   660.2   &   17.5222 &  17.6307 &  17.5233 &   492  &  872.0 & $\geq 10$      \\ 
 p\_hat1000\_2    &  54.8464 &  54.8586 &  54.8522 &  1919 &  2583.9   &   54.8440 &  55.0811 &  54.8723 &  2013  & $\geq 3600$ & $\geq 46$      \\ 
 p\_hat1000\_3    &  83.5297 &  83.5515 &  83.5469 &  1457 &  1954.8   &   83.5285 &  83.5677 &  83.5367 &  1403  & 2460.7 & $\geq 68$      \\  
     san1000      &  15.0003 &  15.1126 &  15.0016 &  1757 &  2420.4   &   14.9999 &  15.0527 &  15.0282 &   893  & 1538.7 &  15   \\	    
 hamming10\_2     &  5819.50 &  5819.50 &  5809.20 &  2141 &$\geq 3600$&   512.220 &  512.220 &  512.220 &  1919  & $\geq 3600$ & 512   \\	    
 hamming10\_4     &  42.6673 &  42.6755 &  42.6678 &   576 &   947.2   &   42.6670 &  42.7206 &  42.6674 &   107  &  198.9 &  40   \\	    
    MANN\_a45     &  14869.8 &  14869.8 &  14869.8 &  2289 &$\geq 3600$&   356.048 &  356.070 &  356.055 &   873  & 1670.5 & 345   \\	    
 p\_hat1500\_1    &  21.8947 &  22.0632 &  21.9062 &   563 &$\geq 3600$&   21.8925 &  22.6598 &  21.8976 &   486  & $\geq 3600$ &  12   \\	    
 p\_hat1500\_2    &  75.9170 &  84.6770 &  89.5774 &   562 &$\geq 3600$&   76.4810 &  76.6852 &  76.6293 &   479  & $\geq 3600$ &  65   \\	  
 p\_hat1500\_3    &  303.073 &  374.111 &  342.539 &   562 &$\geq 3600$&   113.701 &  113.981 &  113.779 &   491  & $\geq 3600$ &  94   \\	  
    MANN\_a81     &  859.136 &  2131.14 &  1467.70 &    50 &$\geq 3600$&   4127.67 &  4127.80 &  4127.67 &    45  & $\geq 3600$ & 1100  \\	  
     keller6      &  3655.24 &  3661.45 &  3637.22 &    47 &$\geq 3600$&   97.8409 &  322.341 &  103.789 &    43  & $\geq 3600$ & $\geq 59$       \\

 \hline

\end{longtable}
}
\end{center}

\begin{figure}[ht]
  \begin{center}
   \includegraphics[trim={1.5cm 0.5cm 1cm 0},clip,width=0.48\textwidth]{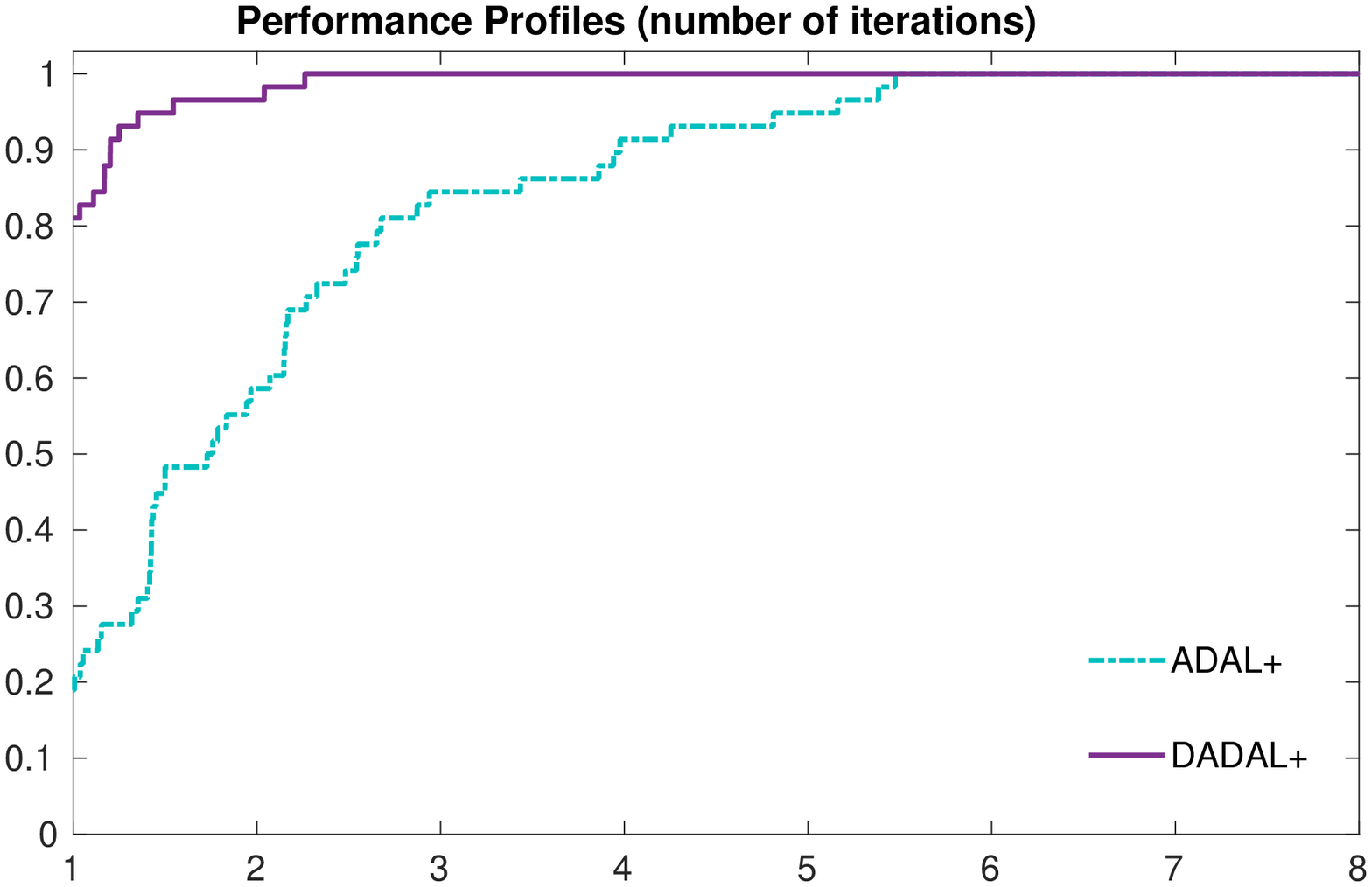} \includegraphics[trim={1.5cm 0.5cm 1cm 0},clip,width=0.48\textwidth]{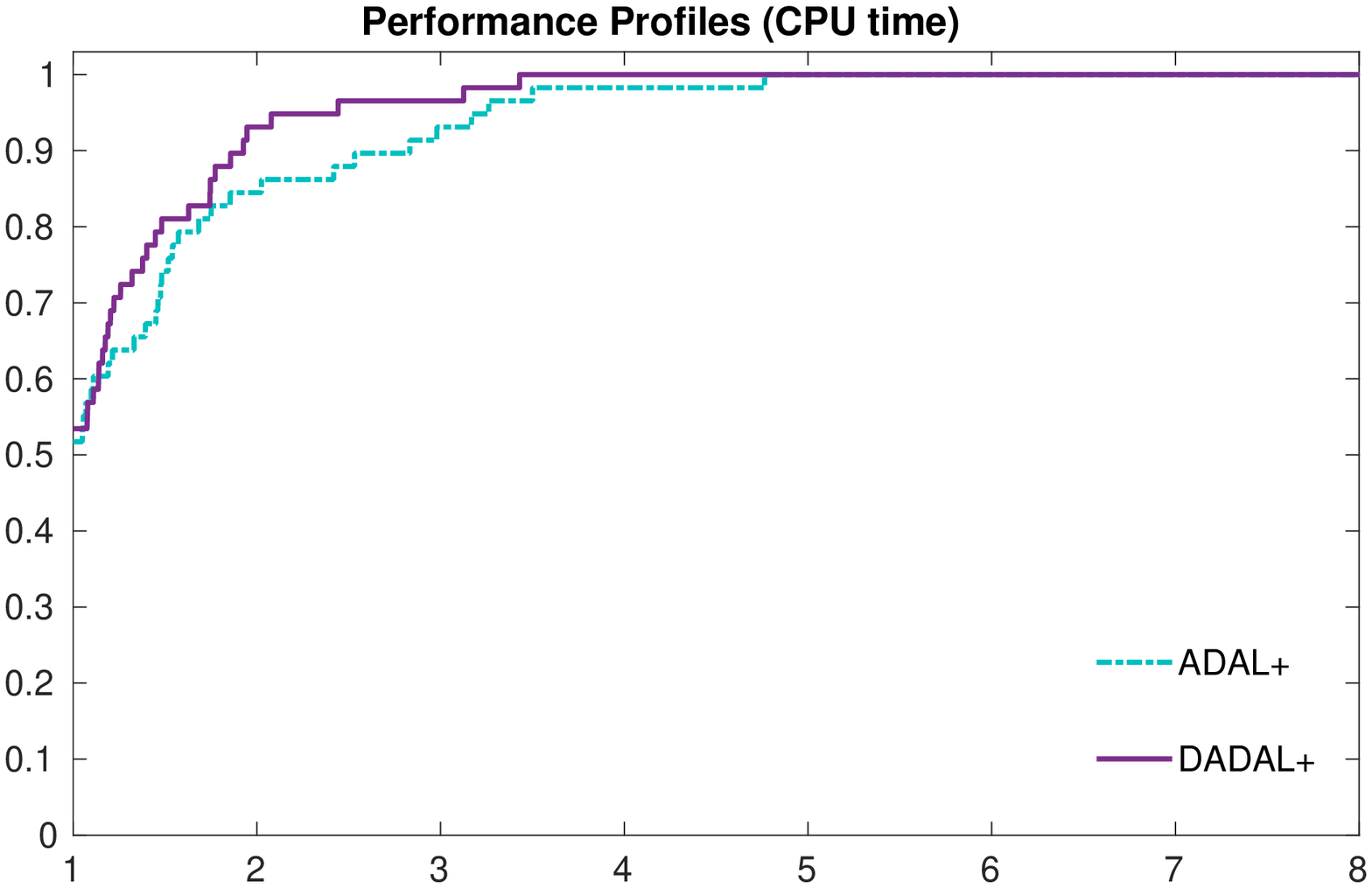}
  \end{center}
  \caption{Comparison between \texttt{ADAL+} and \texttt{DADAL+} on DIMACS instances~\cite{Johnson:1996}.}
  \label{fig:PP-Adal}
\end{figure}
 
\subsection{Comparison between \texttt{ConicADMM3c} and \texttt{DADMM3c}}
In Table~\ref{tab:ConicDADMM3c} we report the results obtained with 
\texttt{ConicADMM3c} and \texttt{DADMM3c}
on the $66$ instances of computing $\vartheta_+(G)$ detailed in Table~\ref{tab:instances}.

As before, we report the name of the instances, the stability number and, for each algorithm, the dual objective function value obtained,  the bounds obtained by computing the error bound and by applying the Nightjet procedure,
the number of iterations and the CPU time needed to satisfy the stopping criterion.

\texttt{ConicADMM3c} was not able to stop within the time limit on $11$ instances, while
\texttt{DADMM3c} was not able to stop within the time limit on $15$ instances.

In general, \texttt{DADMM3c} needs to perform  much less iterations and 
it is slightly better than \texttt{ConicADMM3c} in terms of CPU time as it is confirmed by the performance profiles shown in Figure~\ref{fig:PP-ADMM}.
As before, we did not include the instances that exceeded time limit in the performance profiles.

Again, the Nightjet procedure is able to obtain better bounds with respect 
to the error bounds, both when applied as post-processing phase for \texttt{ConicADMM3c}
and for \texttt{DADMM3c}, for the majority of the instances.
However, there exist cases ($6$ instances) where the Nightjet procedure fails.

We finally mention that on several instances where the time limit was exceeded, the bounds obtained by \texttt{DADMM3c}
are much better than those obtained by \texttt{ConicADMM3c},  
see for example the instances \texttt{p\_hat1500\_1}, \texttt{p\_hat1500\_2} and \texttt{p\_hat1500\_3}.

\begin{center}
\ifFourOR 
\small{
\else 
\scriptsize{
\fi 
\setlength{\tabcolsep}{4.5pt}
\renewcommand{\arraystretch}{1.05}
\begin{longtable}{|l|rrrrr|rrrrr|r|}
\caption{Comparison between \texttt{ConicADMM3c} and \texttt{DADMM3c} on DIMACS instances~\cite{Johnson:1996}.}\label{tab:ConicDADMM3c}\\
\hline
\footnotesize
& \multicolumn{5}{c|}{ \texttt{ConicADMM3c}} & \multicolumn{5}{c|}{ \texttt{DADMM3c}} &\\
 Problem & d ofv & $EB$ & $NB$ & it & time &  d ofv  & $EB$ & $NB$ & it & time & $\alpha$ \\ 
 \endfirsthead
 \multicolumn{12}{c}{{\tablename\ \thetable{} -- continued from previous page}} \\\hline
& \multicolumn{5}{c|}{ \texttt{ConicADMM3c}} & \multicolumn{5}{c|}{ \texttt{DADMM3c}} & \\
 Problem & d ofv & $EB$ & $NB$ & it & time &  d ofv & $EB$ & $NB$ & it & time & $\alpha$ \\ 
 \hline
\hline
 \endhead
 \hline \multicolumn{12}{r}{{continued on next page}} \\
 \endfoot
 \hline
 \endlastfoot
\hline
\hline
johnson8\_2\_4  	&	3.99997	&	4.00033	&	4.00000	&	53	&	0.1	&	  4.00000	&	4.00017	&	4.00010	& 	28	&	0.1&	4	\\
     MANN\_a9   	&	17.4752	&	17.4752	&	17.4752	&	277	&	0.5	&	  17.4750	&	17.4756	&	17.4755	&	988	&	4.4&	16	\\
  hamming6\_2   	&	31.9996	&	32.0004	&	32.0000	&	776	&	2.6	&	  32.0001	&	32.0001	&	32.0001	&	309	&	2.3&	32	\\
  hamming6\_4   	&	4.00002	&	4.00025	&	4.00002	&	69	&	0.2	&	  4.00000	&	4.00145	&	4.00020	&	50	&	0.3&	4	\\
johnson8\_4\_4  	&	13.9999	&	14.0013	&	14.0002	&	151	&	0.5	&	  14.0000	&	14.0000	&	14.0007	&	286	&	2.4&	14	\\
johnson16\_2\_4 	&	7.99995	&	8.00133	&	8.00000	&	106	&	1.0	&	  8.00000	&	8.00404	&	8.00044	&	66	&	0.9&	8	\\
     keller4    	&	13.4660	&	13.4711	&	13.4670	&	338	&	7.1	&	  13.4659	&	13.4703	&	13.4660	&	552	&	16.9&	11	\\
  brock200\_1   	&	27.1967	&	27.2002	&	27.2004	&	332	&	9.6	&	  27.1967	&	27.2020	&	27.1982	&	317	&	12.9&	21	\\
  brock200\_2   	&	14.1310	&	14.1371	&	14.1324	&	199	&	5.9	&	  14.1310	&	14.1380	&	14.1320	&	231	&	10.5&	12	\\
  brock200\_3   	&	18.6718	&	18.6771	&	18.6734	&	226	&	6.8	&	  18.6718	&	18.6778	&	18.6729	&	286	&	12.2&	15	\\
  brock200\_4   	&	21.1210	&	21.1257	&	21.1248	&	260	&	8.4	&	  21.1211	&	21.1269	&	21.1223	&	222	&	9.8&	17	\\
  c\_fat200\_1  	&	11.9999	&	12.0025	&	12.0004	&	261	&	8.6	&	  12.0001	&	12.0004	&	12.0004	&	165	&	7.8&	12	\\
  c\_fat200\_2  	&	24.0017	&	24.0017	&	24.0017	&	1686	&	54.2	&	  23.9981	&	24.0000	&	24.0002	&	974	&	45.3&	24	\\
  c\_fat200\_5  	&	60.3461	&	60.3461	&	60.3461	&	1469	&	45.0	&	  60.3444	&	60.3481	&	60.3463	&	996	&	46.2&	58	\\
san200\_0\_7\_1 	&	30.0019	&	30.0019	&	30.0019	&	2039	&	65.6	&	  30.0019	&	30.0022	&	30.0019	&	1161	&	47.7&	30	\\
san200\_0\_7\_2 	&	17.9991	&	18.0001	&	18.0012	&	8265	&	246.6	&	  17.9421	&	18.0121	&	18.0131	&	87753	&	$\geq 3600$	&	18	\\
san200\_0\_9\_1 	&	70.0019	&	70.0019	&	70.0019	&	2472	&	77.8	&	  69.9980	&	70.0000	&	70.0001	&	1113	&	40.9&	70	\\
san200\_0\_9\_2 	&	60.0017	&	60.0017	&	60.0017	&	2340	&	74.3	&	  59.9980	&	60.0000	&	60.0001	&	1098	&	40.7&	60	\\
san200\_0\_9\_3 	&	43.9983	&	44.0000	&	44.0000	&	11443	&	342.1	&	  43.9984	&	44.0005	&	44.0002	&	10892	&	401.7&	44	\\
  san200\_0\_7  	&	23.6332	&	23.6370	&	23.6364	&	311	&	9.7	&	  23.6333	&	23.6390	&	23.6347	&	323	&	12.1&	18	\\
  san200\_0\_9  	&	48.9044	&	48.9070	&	48.9084	&	686	&	21.5	&	  48.9046	&	48.9049	&	48.9050	&	1659	&	61.1&	42	\\
  hamming8\_2   	&	127.998	&	128.002	&	128.000	&	4163	&	245.5	&	  11.9202	&	434.728	&	Inf	&	52831	&	$\geq 3600$	&	128	\\
  hamming8\_4   	&	15.9997	&	16.0119	&	16.0006	&	168	&	9.8	&	  15.9998	&	16.0087	&	15.9999	&	266	&	18.5&	16	\\
  p\_hat300\_1  	&	10.0203	&	10.0356	&	10.0213	&	404	&	31.6	&	  10.0203	&	10.0339	&	10.0211	&	463	&	46.3&	8	\\
  p\_hat300\_2  	&	26.7141	&	26.7146	&	26.7142	&	3425	&	261.5	&	  26.7167	&	26.7233	&	26.7197	&	1102	&	107.2&	25	\\
  p\_hat300\_3  	&	40.7008	&	40.7098	&	40.7032	&	822	&	63.7	&	  40.7012	&	40.7092	&	40.7043	&	832	&	77.5&	36	\\
    MANN\_a27   	&	132.762	&	132.765	&	132.763	&	7179	&	914.3	&	  26.9452	&	487.963	&	Inf	&	22536	&	$\geq 3600$	&	126	\\
  brock400\_1   	&	39.3308	&	39.3434	&	39.3381	&	528	&	80.7	&	  39.3309	&	39.3455	&	39.3331	&	331	&	61.1&	27	\\
  brock400\_2   	&	39.1964	&	39.2079	&	39.2024	&	532	&	79.9	&	  39.1965	&	39.2113	&	39.1987	&	367	&	67.4&	29	\\
  brock400\_3   	&	39.1602	&	39.1734	&	39.1675	&	526	&	81.4	&	  39.1604	&	39.1749	&	39.1628	&	335	&	61.7&	31	\\
  brock400\_4   	&	39.2313	&	39.2446	&	39.2361	&	515	&	78.6	&	  39.2314	&	39.2459	&	39.2339	&	370	&	69.3&	33	\\
san400\_0\_5\_1 	&	13.0015	&	13.0015	&	13.0015	&	5691	&	853.4	&	  13.0035	&	13.0056	&	13.0035	&	5878	&	1152.4&	13	\\
san400\_0\_7\_1 	&	39.9966	&	40.0000	&	40.0000	&	3961	&	601.7	&	  39.9961	&	40.0000	&	40.0001	&	2663	&	498.9&	40	\\
san400\_0\_7\_2 	&	29.9991	&	30.0000	&	30.0004	&	6704	&	996.0	&	  29.9967	&	30.0020	&	30.0000	&	4459	&	823.8&	30	\\
san400\_0\_7\_3 	&	22.0000	&	22.0063	&	22.0009	&	962	&	147.5	&	  22.0000	&	22.0110	&	22.0015	&	253	&	46.1&	22	\\
san400\_0\_9\_1 	&	100.003	&	100.003	&	100.003	&	4803	&	719.6	&	  99.9961	&	100.000	&	100.000	&	2520	&	440.6&	100	\\
 sanr400\_0\_5  	&	20.1782	&	20.1934	&	20.1837	&	282	&	42.6	&	  20.1782	&	20.1975	&	20.1797	&	332	&	66.4&	13	\\
 sanr400\_0\_7  	&	33.9666	&	33.9790	&	33.9726	&	434	&	66.1	&	  33.9666	&	33.9825	&	33.9685	&	347	&	63.2&	21	\\
johnson32\_2\_4 	&	16.0000	&	16.0000	&	16.0000	&	769	&	194.9	&	  16.0000	&	16.0309	&	16.0006	&	93	&	29.0&	16	\\
 c\_fat500\_10  	&	125.997	&	126.003	&	126.002	&	6666	&	1729.1	&	  168.342	&	170.112	&	168.580 &	10895	&	$\geq 3600$	&	126	\\
  c\_fat500\_1  	&	13.9997	&	14.0050	&	14.0007	&	478	&	132.3	&	  13.9995	&	14.0063	&	14.0005	&	300	&	103.5&	14	\\
  c\_fat500\_2  	&	26.0000	&	26.0034	&	26.0012	&	717	&	192.3	&	  25.9994	&	26.0076	&	26.0009	&	384	&	134.7&	26	\\
  c\_fat500\_5  	&	63.9975	&	64.0029	&	64.0020	&	3752	&	966.2	&	  64.0027	&	64.0030	&	64.0027	&	1684	&	562.4&	64	\\
  p\_hat500\_1  	&	13.0081	&	13.0401	&	13.0103	&	449	&	118.6	&	  13.0080	&	13.0383	&	13.0092	&	631	&	208.9&	9	\\
  p\_hat500\_2  	&	38.5599	&	38.5606	&	38.5600	&	4417	&	1138.1	&	  38.5649	&	38.5767	&	38.5891	&	1356	&	437.8&	36	\\
  p\_hat500\_3  	&	57.8119	&	57.8307	&	57.8198	&	1312	&	340.3	&	  57.8137	&	57.8320	&	57.8457	&	1128	&	349.1&	 $\geq 50$     	\\
  p\_hat700\_1  	&	15.0453	&	15.0996	&	15.0475	&	492	&	314.4	&	  15.0454	&	15.0948	&	15.0504	&	756	&	593.9&	11	\\
  p\_hat700\_2  	&	48.4405	&	48.4416	&	48.4407	&	4971	&	3074.1	&	  48.4470	&	48.4608	&	48.4739	&	1765	&	1355.9&	44	\\
  p\_hat700\_3  	&	71.7561	&	71.7819	&	71.7968	&	2256	&	1413.6	&	  71.7635	&	71.7882	&	71.8177	&	1372	&	1018.2&	62	\\
     keller5    	&	30.9956	&	31.0441	&	30.9980	&	2301	&	1988.5	&	  30.9957	&	31.0352	&	30.9958	&	3470	&	$\geq 3600$	&	27	\\
  brock800\_1   	&	41.8673	&	41.9094	&	41.8704	&	805	&	836.6	&	  41.8674	&	41.9137	&	41.8675	&	370	&	436.4&	23	\\
  brock800\_2   	&	42.1043	&	42.1456	&	42.1069	&	808	&	849.2	&	  42.1043	&	42.1512	&	42.1044	&	382	&	449.6&	24	\\
  brock800\_3   	&	41.8825	&	41.9255	&	41.8903	&	801	&	839.9	&	  41.8826	&	41.9291	&	41.8826	&	374	&	443.5&	25	\\
  brock800\_4   	&	42.0006	&	42.0433	&	42.0058	&	805	&	846.1	&	  42.0006	&	42.0470	&	42.0007	&	375	&	444.1&	26	\\
 p\_hat1000\_1  	&	17.5223	&	17.5606	&	17.5261	&	773	&	2124.9	&	  17.5226	&	17.6075	&	17.5242	&	1071	&	3400.0&	 $\geq 10$     	\\
 p\_hat1000\_2  	&	54.7944	&	55.5397	&	56.0385	&	1317	&	$\geq 3600$	&	54.8680	&	54.9231	&	54.9417	&	1168	&	$\geq 3600$	&	 $\geq 46$     	\\
 p\_hat1000\_3  	&	7663.81	&	7663.82	&	7770.75	&	1327	&	$\geq 3600$	&	83.5456	&	83.5930	&	83.6149	&	1188	&	$\geq 3600$	&	 $\geq 68$     	\\
     san1000    	&	81.8129	&	82.3795	&	84.5386	&	1282	&	$\geq 3600$	&	15.0315	&	15.1395	&	15.0506	&	1144	&	$\geq 3600$	&	15	\\
 hamming10\_2   	&	314167	&	314167	&	314167	&	1154	&	$\geq 3600$	&	65.1578	&	2179.24	&	Inf	&	1024	&	$\geq 3600$	&	512	\\
 hamming10\_4   	&	1653.80	&	1653.8	&	1653.80	&	1117	&	$\geq 3600$	&	42.6662	&	42.7279	&	42.6680	&	265	&	881.6&	40	\\
    MANN\_a45   	&	181901	&	181901	&	181901	&	1159	&	$\geq 3600$	&	36.1173	&	1117.01	&	Inf	&	1061	&	$\geq 3600$	&	345	\\
 p\_hat1500\_1  	&	2850.61	&	2850.61	&	2850.61	&	280	&	$\geq 3600$	&	22.0204	&	22.6797	&	22.0222	&	262	&	$\geq 3600$	&	12	\\
 p\_hat1500\_2  	&	21212.1	&	21212.1	&	21212.1	&	280	&	$\geq 3600$	&	76.7168	&	78.3033	&	76.7352	&	268	&	$\geq 3600$	&	65	\\
 p\_hat1500\_3  	&	41630.4	&	41630.4	&	41630.4	&	282	&	$\geq 3600$	&	114.243 &	116.135	&	114.301 &	271	&	$\geq 3600$	&	94	\\
    MANN\_a81   	&	388.942	&	3282.73	&	2205.46	&	25	&	$\geq 3600$	&	93.4902	&	3269.63	&	Inf	&	24	&	$\geq 3600$	&	1100	\\
     keller6    	&	1902.86	&	1902.86	&	   Inf 	&	22	&	$\geq 3600$	&	181.517	&	184.845	&	197.191	&	24	&	$\geq 3600$	&	 $\geq 59$ 	\\
 \hline
\end{longtable}
}
\end{center}

\begin{figure}[ht]
  \begin{center}
   \includegraphics[trim={1.5cm 0.5cm 1cm 0},clip,width=0.48\textwidth]{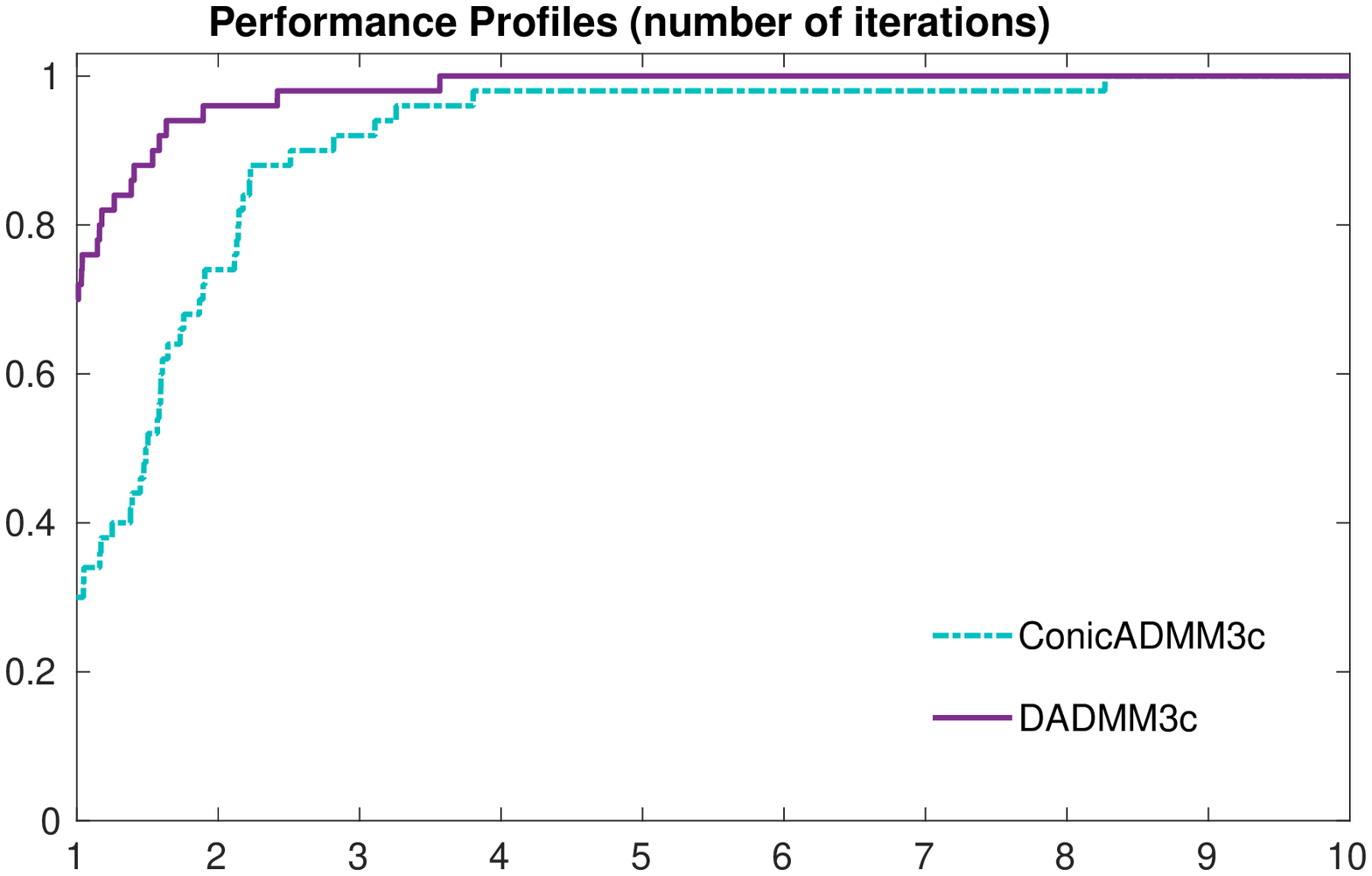} \includegraphics[trim={1.5cm 0.5cm 1cm 0},clip,width=0.48\textwidth]{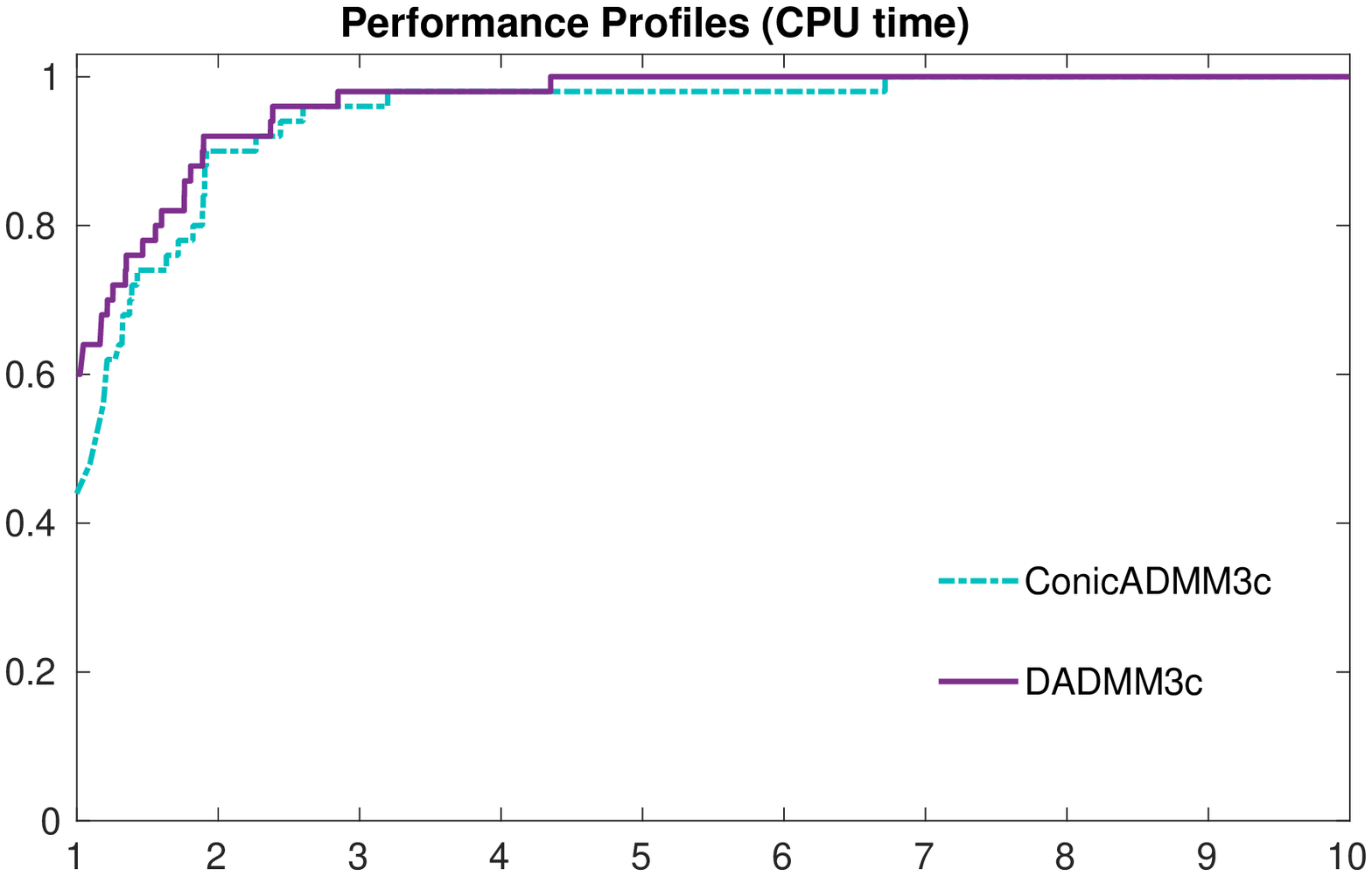}
  \end{center}
  \caption{Comparison between \texttt{ConicADMM3c} and \texttt{DADMM3c} on DIMACS instances~\cite{Johnson:1996}.}
  \label{fig:PP-ADMM}
\end{figure}

\section{Conclusions}\label{sec:conclusions}
In this paper we propose to use a factorization of the dual matrix within two 
ADMMs for conic programming proposed 
in the literature. 
In particular we use a first order update of the dual variables in order 
to improve the performance of the ADMMs considered.

Our computational results on instances from a DNN relaxation of the stable set problem show 
that the factorization employed gives a significant improvement in the 
efficiency of the methods. We are confident that this can be the case also when 
dealing with other structured DNNs. 
In particular, we experience a drastic reduction in 
terms of number of iterations.
The performance of \texttt{DADMM3c} may even further improve through a smart 
update of the rank of $Z$ along the iterations. This is a topic for future 
investigation.

In the paper we also focus on how to obtain bounds on the primal optimal 
objective function value, since the dual objective function value obtained when using first order 
methods to solve DNNs is not always guaranteed to serve as bound, as the dual solution may be infeasible. 
We present two methods: one that adds a sufficient (negative) perturbation to the dual objective function value (error bounds) and one that constructs a dual feasible solution (Nightjet procedure). 
Both methods are computationally cheap and produce bounds close to the optimal objective function value of the DNN if the obtained solution is close to the optimal solution. 
The Nightjet procedure works particularly well for structured instances, like computing $\vartheta_+$, but comes with the drawback that it might fail to produce a feasible solution. However, as long as the dual solution is reasonably close to the (unknown) optimal solution, this does not happen. 
We also observe that the Nightjet procedure works particularly well after \texttt{ADAL+} and \texttt{DADAL+}. This is due to the fact that in these algorithms the dual matrix (which is the input for the Nightjet procedure) is positive semidefinite by construction.
The two versions of the post-processing make our methods applicable 
within branch-and-bound frameworks in order to solve combinatorial 
optimization problems with DNN relaxations.  

Our plan for future research is to apply the methods to other structured DNN relaxations. Furthermore, we will expand our methods to solve SDPs with general inequality constraints instead of just nonnegativity.

\ifFourOR 
\begin{acknowledgements}
\else 
\section*{Acknowledgements}
\fi 

We thank Kim-Chuan Toh for bringing our attention to~\cite{JaChayKeil2007} and 
for providing an implementation of the method therein.
We also thank the Nightjet NJ 40233 from Klagenfurt to Roma Termini on November 12, 
2019 for being one hour delayed; this helped Elisabeth Gaar to focus and set 
the ground stone for the Nightjet procedure. 
We also thank two anonymous referees for their valuable comments that helped us to improve the paper.
\ifFourOR 
\end{acknowledgements}
\else 
\fi 

\ifFourOR 

%
\section*{Conflict of interest}
The authors declare that they have no conflict of interest.


\bibliographystyle{spmpsci} 

\else 
\bibliographystyle{plain} 
\fi 

\bibliography{DADALplus}

\def\cprime{$'$} \def\cprime{$'$}
\begin{thebibliography}{10}

\bibitem{Be:82}
Dimitri~P. Bertsekas.
\newblock {\em Constrained optimization and {L}agrange multiplier methods}.
\newblock Computer Science and Applied Mathematics. Academic Press Inc., New
  York, 1982.

\bibitem{chen2016direct}
Caihua Chen, Bingsheng He, Yinyu Ye, and Xiaoming Yuan.
\newblock The direct extension of {ADMM} for multi-block convex minimization
  problems is not necessarily convergent.
\newblock {\em Mathematical Programming}, 155(1-2):57--79, 2016.

\bibitem{deSaReWie:2018}
Marianna {De Santis}, Franz Rendl, and Angelika Wiegele.
\newblock Using a factored dual in augmented {L}agrangian methods for
  semidefinite programming.
\newblock {\em Operations Research Letters}, 46(5):523 -- 528, 2018.

\bibitem{GaarRendl}
Elisabeth Gaar and Franz Rendl.
\newblock A bundle approach for {SDP}s with exact subgraph constraints.
\newblock In Andrea Lodi and Viswanath Nagarajan, editors, {\em Integer
  Programming and Combinatorial Optimization}, pages 205--218. Springer
  International Publishing, 2019.

\bibitem{ThetaValues_GiLeRoSm:13}
Monia Giandomenico, Adam~N. Letchford, Fabrizio Rossi, and Stefano Smriglio.
\newblock Approximating the {L}ov{\'a}sz $\theta$ function with the subgradient
  method.
\newblock {\em Electronic Notes in Discrete Mathematics}, 41:157 -- 164, 2013.

\bibitem{StableSetHardToApprox}
Johan H{\aa}stad.
\newblock Clique is hard to approximate within $n^{1 - \varepsilon}$.
\newblock {\em Acta Mathematica}, 182(1):105--142, 1999.

\bibitem{JaChayKeil2007}
Christian Jansson, Denis Chaykin, and Christian Keil.
\newblock Rigorous error bounds for the optimal value in semidefinite
  programming.
\newblock {\em SIAM Journal on Numerical Analysis}, 46(1):180--200, 2007/08.

\bibitem{Johnson:1996}
David~J. Johnson and Michael~A. Trick, editors.
\newblock {\em Cliques, Coloring, and Satisfiability: Second DIMACS
  Implementation Challenge, Workshop, October 11-13, 1993}.
\newblock American Mathematical Society, 1996.

\bibitem{karpStableSetColoringNP}
Richard~M. Karp.
\newblock Reducibility among combinatorial problems.
\newblock In Raymond~E. Miller, James~W. Thatcher, and Jean~D. Bohlinger,
  editors, {\em Complexity of Computer Computations}, The IBM Research Symposia
  Series, pages 85--103. Springer US, 1972.

\bibitem{Lorenz2019}
Dirk~A. Lorenz and Quoc Tran-Dinh.
\newblock Non-stationary {D}ouglas--{R}achford and alternating direction method
  of multipliers: adaptive step-sizes and convergence.
\newblock {\em Computational Optimization and Applications}, 74(1):67--92, Sep
  2019.

\bibitem{MaPoReWi:09}
J{\'e}r{\^o}me Malick, Janez Povh, Franz Rendl, and Angelika Wiegele.
\newblock Regularization methods for semidefinite programming.
\newblock {\em SIAM J. Optim.}, 20(1):336--356, 2009.

\bibitem{matRel}
Franz Rendl.
\newblock Matrix relaxations in combinatorial optimization.
\newblock In Jon Lee and Sven Leyffer, editors, {\em Mixed Integer Nonlinear
  Programming}, pages 483--511. Springer New York, 2012.

\bibitem{SchrijverNonNeg}
Alexander Schrijver.
\newblock A comparison of the {D}elsarte and {L}ov{\'a}sz bounds.
\newblock {\em IEEE Transactions on Information Theory}, 25(4):425--429, 1979.

\bibitem{Sun2015AC3}
Defeng Sun, Kim-Chuan Toh, and Liuqin Yang.
\newblock A convergent 3-block semiproximal alternating direction method of
  multipliers for conic programming with 4-type constraints.
\newblock {\em SIAM Journal on Optimization}, 25:882--915, 2015.

\bibitem{Wen2010}
Zaiwen Wen, Donald Goldfarb, and Wotao Yin.
\newblock Alternating direction augmented {L}agrangian methods for semidefinite
  programming.
\newblock {\em Mathematical Programming Computation}, 2(3):203--230, 2010.

\end{thebibliography}

\end{document}